\documentclass[fleqn,11pt,twoside]{article}
\usepackage{
 color, xcolor,epsfig, graphics, subfigure}
\usepackage{amsfonts,amsthm,amsmath,amssymb,
	amsbsy,euscript}

\newcommand{\by}[1]{\textrm{{#1}}}
\newcommand{\jour}[1]{\textit{{#1}}}
\newcommand{\vol}[1]{V.\textrm{{#1}}}
\newcommand{\book}[1]{\textrm{{#1}}}

\makeatletter
\newcommand{\copyrightnote}[2]{{\renewcommand{\thefootnote}{}
 \footnotetext{\small\it
\begin{flushleft}
 \copyright \ #1   #2  
\end{flushleft}}}}

\newcommand{\Name}[1]{\begin{flushleft}
                       \LARGE \bf #1
                       \end{flushleft}\vspace{-3mm}}

\newcommand{\Author}[1]{\begin{flushleft}
                       \it #1 \end{flushleft}}

\newcommand{\Address}[1]{\begin{flushleft}
                       \it #1 \end{flushleft}}

\newcommand{\Date}[1]{\begin{flushleft}
                      \small  \it #1 \end{flushleft}}

\newcommand{\evenhead}{Author \ name}
\newcommand{\oddhead}{Article \ name}

\renewcommand{\@evenhead}{
\hspace*{-3pt}\raisebox{-15pt}[\headheight][0pt]{\vbox{\hbox to \textwidth
{\thepage \hfil \evenhead}\vskip4pt \hrule}}}
\renewcommand{\@oddhead}{
\hspace*{-3pt}\raisebox{-15pt}[\headheight][0pt]{\vbox{\hbox to \textwidth
{\oddhead \hfil \thepage}\vskip4pt\hrule}}}
\renewcommand{\@evenfoot}{}
\renewcommand{\@oddfoot}{}

\long\def\@makecaption#1#2{%
  \vskip\abovecaptionskip
  \sbox\@tempboxa{\small \textbf{#1.}\ \ #2}%
  \ifdim \wd\@tempboxa >\hsize
    {\small \textbf{#1.}\ \ #2}\par
  \else
    \global \@minipagefalse
    \hb@xt@\hsize{\hfil\box\@tempboxa\hfil}%
  \fi
  \vskip\belowcaptionskip}

\newcommand{\JNMPnumberwithin}[3][\arabic]{%
  \@ifundefined{c@#2}{\@nocounterr{#2}}{%
    \@ifundefined{c@#3}{\@nocnterr{#3}}{%
      \@addtoreset{#2}{#3}%
      \@xp\xdef\csname the#2\endcsname{%
        \@xp\@nx\csname the#3\endcsname .\@nx#1{#2}}}}%
}

\renewenvironment{proof}[1][\proofname]{\par
  \normalfont
  \topsep6\p@\@plus6\p@ \trivlist
  \item[\hskip\labelsep\textbf{%
    #1\@addpunct{.}}]\ignorespaces
}{%
  \qed\endtrivlist
}

\newcommand{\resetfootnoterule} {
  \renewcommand\footnoterule{%
  \kern-3\p@
  \hrule\@width.4\columnwidth
  \kern2.6\p@}
}

\renewcommand{\footnoterule}{}

\makeatother

\newtheorem{theor}{Theorem}

\theoremstyle{definition}

\newtheorem{proposition}[theor]{Proposition}
\newtheorem{lemma}[theor]{Lemma}
\newtheorem{cor}[theor]{Corollary}
\newtheorem{define}{Definition}

\newtheorem{problem}{Problem}
\newtheorem{open}[problem]{Open problem}
\newtheorem{example}{Example}
\newtheorem{counterexample}[example]{Counterexample}
\theoremstyle{remark}
\newtheorem{rem}{Remark}

\def\oldvec{\mathaccent "017E\relax }
\DeclareMathOperator{\Or}{\mathsf{O\oldvec{r}}}

\newcommand{\pinner}{\mathbin{\mathchoice
{\hbox{\vrule width0.6em depth0pt height0.4pt
	\vrule width0.4pt depth0pt height0.8ex}}
{\hbox{\vrule width0.6em depth0pt height0.4pt
	\vrule width0.4pt depth0pt height0.8ex}}
{\hbox{\kern0.14em
	\vrule width0.48em depth0pt height0.4pt
	\vrule width0.4pt depth0pt height0.6ex\kern0.14em}}
{\hbox{\kern0.1em
	\vrule width0.39em depth0pt height0.4pt
	\vrule width0.4pt depth0pt height0.5ex\kern0.1em}}}}

\newcommand{\BBR}{\mathbb{R}}
\newcommand{\BBN}{\mathbb{N}}
\newcommand{\BBS}{\mathbb{S}}
\newcommand{\BBE}{\mathbb{E}}

\newcommand{\cE}{\mathcal{E}}

\newcommand{\cP}{\mathcal{P}}

\newcommand{\boldb}{{\boldsymbol{b}}}

\newcommand{\ba}{{\boldsymbol{a}}}

\newcommand{\bx}{{\boldsymbol{x}}}
\newcommand{\bX}{{\mathbf{x}}}
\newcommand{\bby}{{\boldsymbol{y}}}

\newcommand{\veps}{\varepsilon}

\newcommand{\dd}{\partial}

\newcommand{\ld}{{\text{d}}}

\DeclareMathOperator{\id}{id}

\DeclareMathOperator{\dvol}{d
vol}

\newcommand{\lshad}{[\![}
\newcommand{\rshad}{]\!]}

\newcommand{\schouten}[1]{\lshad {#1} \rshad}

\newcommand{\ground}[1]{\text{\textit{\small #1}}}

\makeatletter
\DeclareRobustCommand{\square}{\mathbin{\mathpalette\morphic@square\relax}}
\newcommand{\morphic@square}[2]{%
  \sbox\z@{$\m@th#1\rule{4pt}{4pt}$}%
  \vcenter{\box\z@}%
}
\makeatother

\begin{document}

\renewcommand{\evenhead}{ {\LARGE\textcolor{blue!10!black!40!green}{{\sf \ \ \ ]ocnmp[}}}\strut\hfill 
R.~Buring, D.~Lipper and A.\,V.\,Kiselev
}
\renewcommand{\oddhead}{ {\LARGE\textcolor{blue!10!black!40!green}{{\sf ]ocnmp[}}}\ \ \ \ \   
The hidden symmetry of Kontsevich's graph flows
}

\thispagestyle{empty}
\newcommand{\FistPageHead}[3]{
\begin{flushleft}
\raisebox{8mm}[0pt][0pt]
{\footnotesize \sf
\parbox{150mm}{{Open Communications in Nonlinear Mathematical Physics}\ \  \ {\LARGE\textcolor{blue!10!black!40!green}{]ocnmp[}}
\ \ Vol.2 (2022) pp
#2\hfill {\sc #3}}}\vspace{-13mm}
\end{flushleft}}

\FistPageHead{1}{\pageref{firstpage}--\pageref{lastpage}}{ \ \ Article}

\strut\hfill

\strut\hfill

\copyrightnote{The author(s). Distributed under a Creative Commons Attribution 4.0 International License}

\Name{The hidden symmetry of Kontsevich's graph flows on the spaces of Nambu\/-\/determinant Poisson brackets%
}

\Author{
R.~Buring$^{\,1,2}$,
D.~Lipper$^{\,3}$
and 
A.\,V.\,Kiselev$^{\,2,3,4}$}

\Address{$^{1}$ Address: 
Institut f\"ur Mathematik, 
Johannes Gutenberg\/--\/Uni\-ver\-si\-t\"at,
Staudingerweg~9, 
\mbox{D-\/55128} Mainz, Germany.
\\[2mm]
$^{2}$ Present address: Institut des Hautes $\smash{\text{\'Etudes}}$ Scientifiques ($\smash{\text{IH\'ES}}$),
35~route de Chartres, Bures\/-\/sur\/-\/Yvette, \mbox{F-91440} France.
\\[2mm]
$^{3}$ Address: Ber\-nou\-lli Institute for Mathematics, Computer Science and Artificial Intelligence, University of Groningen, P.O.~Box 407, 9700~AK Groningen, The Netherlands.
\\[2mm]
$^{4}$ Corresponding author. E-mail: \textup{\texttt{A.V.Kiselev\symbol{"40}rug.nl}}
}

\Date{Received 14 December 2021; Accepted 1 December 2022}

\setcounter{equation}{0}

\begin{abstract}
\noindent%
Kontsevich's graph flows are --\,universally for all finite\/-\/dimensional affine Poisson manifolds\,-- infinitesimal symmetries of the spaces of Poisson brackets. We show that the previously known tetrahedral flow and the recently obtained pentagon\/-\/wheel flow preserve the class of Nambu\/-\/determinant Poisson bi\/-\/vectors 
$P=\lshad \varrho(\bx)\,\dd_x\wedge\dd_y\wedge\dd_z,a\rshad$ on $\BBR^3\ni\bx=(x,y,z)$ and
$P=\lshad\lshad\varrho(\bby)\,\dd_{x^1}\wedge\ldots\wedge\dd_{x^4},a_1\rshad,a_2\rshad$ on $\BBR^4\ni\bby$,
including the general case $\varrho \not\equiv 1$. 
We detect 
that the Poisson bracket evolution $\dot{P} = Q_\gamma(P^{\otimes\operatorname{\# Vert}(\gamma)})$ is trivial in the 
second 
Poisson cohomology, $Q_\gamma = \schouten{P, \smash{\vec{X}}([\varrho],[a])}$, for the Nambu\/-\/determinant bi\/-\/vectors $P(\varrho,[a])$ on~$\BBR^3$. For the global Casimirs $\ba = (a_1$,$\ldots$,$a_{d-2})$ and inverse density $\varrho$ on~$\BBR^d$, we analyse the combinatorics of their evolution induced by the Kontsevich graph flows, namely $\dot{\varrho} = \dot{\varrho}([\varrho], [\ba])$ and $\dot{\ba} = \dot{\ba}([\varrho],[\ba])$ with differential\/-\/polynomial right\/-\/hand sides. Besides the anticipated collapse of these formulas by using the Civita symbols (three for the tetrahedron $\gamma_3$ and five for the pentagon\/-\/wheel graph cocycle $\gamma_5$), as dictated by the behaviour $\varrho(\bX') = \varrho(\bx) \cdot \det \lVert \partial \bX' / \partial \bx \rVert$ of the inverse density $\varrho$ under reparametrizations $\bx \rightleftarrows \bX'$, we discover another, so far hidden discrete symmetry in the construction of these evolution equations.
\end{abstract}
\label{firstpage}


\section{Introduction}
Kontsevich's infinitesimal symmetries $P\mapsto P+\veps Q(P)+\bar{o}(\veps)$ of the spaces of Poisson structures are universal for all finite\/-\/dimensional affine Poisson manifolds $(\mathit{M}_{\text{aff}}^d,P)$, preserving the property of the Cauchy datum $P(\veps=0)$ to remain Poisson modulo $\bar{o}(\veps)$ at~$\veps>0$. 
Differential\/-\/polynomial in the bi\/-\/vector components,
the right\/-\/hand sides~$Q(P)$ of the flows are encoded by using the graph cocycles in the Kontsevich undirected graph complex. The tetrahedral cocycle flow $\dot{P}=\Or(\gamma_3)(P^{\otimes^4})$ is the first example from the pioneering paper~\cite{Ascona96} (cf.~\cite{Bourbaki2017} and~\cite{f16}); graph cocycles beyond the tetrahedron $\gamma_3$ are discussed in~\cite{JNMP17} (see references therein); the next, higher nonlinearity degree flows are constructed for the pentagon\/-\/wheel cocycle~$\gamma_5$ in~\cite{sqs17} and for the heptagon\/-\/wheel cocycle~$\gamma_7$ in~\cite{JPCS17}. We now study the restriction of this universal construction to a particular class of Poisson brackets, so that their analytic properties repercuss in the combinatorics of algebraic structures and in the Poisson-cohomological (non)triviality of the infinitesimal deformations $P\mapsto P+\veps Q(P)+\bar{o}(\veps)$ with the markers $Q\in\text{ker}\schouten{P,\cdot}$ of second Poisson cohomology classes $[Q]\in H^2_P(\mathit{M}_{\text{aff}}^d=\BBR^d)$.

The goal of this paper is to explore the combinatorics that arises for the restriction of these symmetry flows $\dot{P}=Q(P)$ to the class of generalized Nambu\/-\/determinant Poisson brackets,
\begin{equation}\label{EqPBrNambu}
    \{f,g\}=\varrho(\bx)\cdot\det\bigl\lVert\partial(a_1,\ldots,a_{d-2},f,g)/\partial(x^1,\ldots,x^d)\bigr\rVert,
\end{equation}
with $d-2$ global Casimirs $\ba=(a_1,\ldots,a_{d-2})$ and inverse density $\varrho$ on~$\BBR^d$. These bi\/-\/vectors' components, referred to a system of (global, e.g., Cartesian) coordinates, are
\begin{equation}\label{EqBiVectNambu}
P^{ij}= \{x^i,x^j\} = \varrho(\bx)\cdot \sum\nolimits_{i_1,\ldots,i_{d-2}} \veps^{i_1\cdots i_{d-2}ij}\cdot \frac{\dd a_1}{\dd x^{i_1}}\cdots\frac{\dd a_{d-2}}{\dd x^{i_{d-2}}},
\end{equation}
where $\veps^{\vec{\imath}}$ is the Civita symbol.

\begin{example}\label{ExNambuR3}
Among the most well known examples of Poisson structures from this class we recall, for instance,\\[0.5pt] 
$\bullet$\quad 
the Euler top bracket $\{x^i,x^j\}=\veps^{ijk}\cdot x^k$ on $\BBE^3\simeq\textit{so}(3)^*$, that is $\{x,y\}=z$ and so on w.r.t.\ the signed permutations $\sigma\in \BBS_3$. This bracket is Nambu\/-\/class with $\varrho\equiv1$ and the global polynomial Casimir $a(x,y,z)=\tfrac{1}{2}(x^2+y^2+z^2)$.\\[1pt]
$\bullet$\quad the log\/-\/symplectic bracket $\{x,y\}=\tfrac{1}{2}xy$ (and so on, cyclically), given on~$\BBR^3$ with $\varrho\equiv1$ by the Casimir $a=\tfrac{1}{2}xyz$. This bracket is important in deformation quantization (on~$\BBR^2\subset\BBR^3$) since it is expected 
that $x\star y=\text{exp}(\hbar)\cdot y\star x$ for the associative noncommutative star-product with this Poisson bracket, $\{x,y\}=xy$, in the leading deformation term (see \cite{Operads1999,DQonAlgVarieties,MKLefschetzLectures} and~\cite{BPP}).
\end{example}

Linear in the functional parameters $\varrho$ and $\ba$, the Nambu\/-\/determinant 
bi\/-\/vectors~\eqref{EqBiVectNambu}
constitute a large class of Poisson structures which are special in the following sense.
Firstly,   
for any choice of $\varrho\not\equiv1$, Nambu\/-\/determinant Poisson brackets~\eqref{EqPBrNambu}
admit the maximal set of~$d-2$ 
Casimirs $\ba=(a_1$,$\ldots$,$a_{d-2})$. The space~$\BBR^d$ is foliated by the intersections of the level sets $\{a_i=\text{const}\}$ into symplectic leaves, which are generally two\/-\/dimensional: e.g., consider the concentric spheres $\{\bx$~$|$ 
$x^2+y^2+z^2=\text{const}>0\}$ for the Euler top.
In consequence,  
Nambu\/-\/determinant Poisson brackets~\eqref{EqPBrNambu} 
all have rank not exceeding two, so that all the minors of size $3\times3$ (and higher) in its coefficient matrix~\eqref{EqBiVectNambu} vanish for this class of brackets. This is not always so for other types: e.g., consider the nondegenerate symplectic case on~$\BBR^{2d}$.

And not every Poisson bracket on~$\BBR^3$ admits a global polynomial Casimir $a\not\equiv\text{const}$ 
if the coefficients $P^{ij}$ of the bi\/-\/vector $P$ are polynomial. (Whereas for the Nambu class this is achieved tautologically by taking $\varrho,a_i\in\BBR[x^1,\ldots,x^d]$ in any fixed system of affine coordinates on~$\BBR^d$ in any dimension $d\geqslant3$.)

\begin{counterexample}[\cite{MKConjecture13Dec2019adot}]\label{ExPBrNoPolynomCasimir}
On~$\BBR^d$ with Cartesian coordinates $\bx=(x^1,\ldots,x^d)$, consider the Euler vector field $\smash{\vec{E}}=\sum_i x^i\cdot\partial/\partial x^i$ and, for any $k\geqslant2$, take another homogeneous vector field $\vec{V}=\sum_j (x^j)^k\cdot\partial/\partial x^j$. By definition, put $P\mathrel{{:}{=}}\vec{V}\wedge\vec{E}$. Then the bi\/-\/vector $P$ is Poisson --- yet it does not admit any non-constant global polynomial Casimir $a$ on~$\BBR^d$. (A proof is recalled in Appendix~\ref{AppMKPoisson}, see p.~\pageref{AppMKPoisson} below.)
\end{counterexample}

In the same context of competing ``generic vs special", Kontsevich's graph flows provide (markers of the) second Poisson cohomology classes $Q([P])\in\text{ker}\schouten{P,\cdot}$ in an extremely broad setting: indeed, universally for all finite\/-\/dimensional affine Poisson manifolds $(\mathit{M}_{\text{aff}}^d,P)$. This automatically poses the problem of (non)triviality for these Poisson cohomology classes $[Q]\in H_P^2(\mathit{M}_{\text{aff}}^d)$. We recall from \cite{f16,OrMorphism} that for nontrivial graph cocycles~$\gamma$ in the Kontsevich undirected graph complex, there does not exist any mechanism that would trivialize the flows $\dot{P}=\Or(\gamma)(P^{\otimes\#\text{Vert}(\gamma)})$ at the level of Kontsevich's graphs, that is, by using a would\/-\/be universal trivializing vector field $\smash{\vec{X}}$ again determined within the graph language, and hence by a 
formula that would work uniformly in \emph{all} dimensions. 
For instance, such is manifestly the case for the tetrahedron $\gamma_3$, for the pentagon\/-\/wheel cocycle $\gamma_5$, etc.
In other words, the coboundary equation,
\[
\Or(\gamma)(P^{\otimes\#\text{Vert}(\gamma)}) - \schouten{P,\smash{\vec{X}}(\gamma')} = \Diamond(P,\schouten{P,P}),
\]
has no solution $(\gamma',\Diamond)$ in terms of graphs~$\gamma'$ and Leibniz graphs~$\Diamond$
for the main sequence of nontrivial graph cocycles $\gamma_3,\gamma_5,\gamma_7,\ldots$ and their iterated commutators.
Let us emphasize that the Kontsevich graph language from~\cite{Ascona96} is universal for all dimensions $d\geqslant2$ of the Poisson manifold at hand; we detect that for the cocycles~$\gamma_3$, $\gamma_5$, and~$\gamma_7$ the coboundary equation has no solution if the dimension $d\geqslant3$ is not fixed \textit{a priori}.
But if $d=2$, the graph $\gamma'$ trivializing the tetrahedral $\gamma_3$-flow is found in \cite{f16});
likewise, in~\cite[Ch.\,6]{BuringDisser} the trivializing vector fields~$\vec{X}$ over $d=2$ were found for the pentagon\/-\/wheel $\gamma_5$-\/flow and for the $\gamma_7$-\/flow. In all these cases, namely for the graph cocycles $\gamma_3$, $\gamma_5$, and $\gamma_7$, the trivializing vector fields~$\vec{X}$ on~$\BBR^2$ are Hamiltonian w.r.t.\ the standard symplectic structure on~$\BBR^2$ and w.r.t. three Hamiltonians $H_3$, $H_5$, and~$H_7$ which are also available from~\cite[Ch.\,6]{BuringDisser}.
Moreover, when the dimension is fixed to $d=2$, these three Hamiltonians themselves are also realized by using Kontseivich graphs. We conclude that in a given dimension~$d\geqslant2$, there can appear more objects specific to Kontsevich's graph flows over Poisson manifolds of that dimension; these new objects such as the trivializing vector fields~$\vec{X}$ and their Hamiltonians~$H$ can again be realized by using the Kontsevich graph language or its extensions: see~\S\ref{SecVFTrivialFlows} below (where we have $d=3$).
The present work serves to continue --\,from~\cite{f16,BuringDisser,OrMorphism,sqs19}\,-- the line of study on the Poisson (non)\/triviality of Kontsevich's graph flows in arbitrary or prescribed dimension~$d\geqslant3$.

The fact we discover is that for several 
classes of Poisson structures, the Kontsevich graph flows \textit{are} Poisson\/-\/trivial, so that the resulting shifts $Q([P])=\Or(\gamma)(P^{\otimes^n})$ of Poisson bi\/-\/vectors $P$ are induced by highly nonlinear, non\/-\/affine reparametrizations of the base coordinates --\,along the integral trajectories of the trivializing vector fields $\smash{\vec{X}}$\,-- on the \textit{affine} Poisson manifolds~$\mathit{M}_{\text{aff}}^d$. Such is the case for the Nambu\/-\/determinant class of brackets~$P(\varrho,[a])$ on~$\BBR^3$ and the tetrahedral graph flow preserving it. We establish the fact of trivialization and we collapse the formula of the vector field $\smash{\vec{X}}([\varrho],[a])$ by using the features of the Nambu\/--\/Poisson geometry under study. (All these analytic and combinatorial results are verified by direct calculation.) 
We express the vector field~$\smash{\vec{X}}$ in terms of ``micro\/-\/graphs'' the construction of which is specific to dimension $d=3$; it remains to explain the work of trivialization and collapse mechanism in a way which would allow extension to~$d\geqslant3$.

\begin{rem}\label{RemMakeRhoUnit}
For a chosen volume element $\dvol(\bx)=\ld\bx/\varrho(\bx)$ with smooth~$\varrho$, needed for 
construction of the Nambu\/-\/determinant bi\/-\/vectors $P=\ld\ba/\dvol(\bx)$, the zero locus of the inverse density $\varrho$ provides a tiling of the affine space~$\BBR^d$.
Inside each cell bounded by the walls $\{\bx$~$|$ $\varrho(\bx)=0\}$, 
that is on every maximal subset where the restriction of~$\varrho$ is nowhere vanishing,
the inverse density can be brought to a constant $\varrho'(\mathbf{x}')\equiv\pm1$ by a (non)\/linear, pointwise\/-\/dependent rescaling of local coordinates. The restriction of the graph flows to the subclass of `genuine' Nambu\/-\/determinant brackets $P=\ld\ba/\ld\bx$ can either degenerate (e.g., for the tetrahedral flow over~$\BBR^3$) or stay nonzero (e.g., for the tetrahedral flow over~$\BBR^4$), see below. In all these cases, the trivializing vector fields $\smash{\vec{X}}$ behave in a usual way, as tensors do, in the course of such transformations to the normal coordinates; note that the vector fields $\smash{\vec{X}}$ can also acquire arbitrary Poisson\/-\/exact 
summands~$\schouten{P,H}$. Yet the construction of the normal coordinates satisfying $\varrho'(\mathbf{x}')=\pm1$ is \textit{a priori} not correlated at all with the affine structure --- which the graph flows refer to.\\[0.5pt]
\centerline{\rule{1in}{0.7pt}}
\end{rem}

\noindent This paper is structured as follows. In \textsection \ref{SecPrelim} we recall the construction of Nambu\/--\/determinant Poisson brackets $P(\varrho,[\ba])$ on~$\BBR^d$ as derived brackets and we inspect how their elements~$\varrho$ and~$\ba$ are deformed along vector fields~$\vec{X}$ on~$\BBR^d$. We recall also the construction of 
Kontsevich's graph flows on spaces of Poisson structures over affine Poisson manifolds $(\mathit{M}_{\text{aff}}^d,P)$. Next, in \textsection \ref{SecStability} we detect that the Nambu class of Poisson brackets on~$\BBR^3$ and~$\BBR^4$ is preserved by the graph flows for the tetrahedral cocycle~$\gamma_3$ and by the pentagon\/-\/wheel cocycle $\gamma_5$ over~$\BBR^3$. The structure of induced evolution $\dot{\varrho}([\varrho],[\ba]),\dot{a}([\varrho],[\ba])$ is then put, in \textsection \ref{SecHowEvolve}, in correspondence with the original graph cocycle, and the formulas of induced velocities are collapsed by using the Civita symbols (one per graph vertex minus one overall: e.g., three symbols for the tetrahedron); the affine structure of~$\BBR^d$ is crucial at that point. In \textsection \ref{SecMarkers} we analyze the algebra and combinatorics of the marker\/-\/monomials under the sums with multiple Civita symbols. Here we discover an extra symmetry of the Kontsevich graph flows' restriction to the spaces of Nambu\/-\/determinant Poisson structures. Finally, we establish in~\S\ref{SecVFTrivialFlows} that the tetrahedral flow over~$\BBR^3$ is Poisson-cohomology trivial, and we collapse the formula of the trivializing vector field~$\smash{\vec{X}}$ by using the same mechanism of Civita symbols as before. The paper concludes with a list of open problems about the graph flows and combinatorics of their restrictions to the Nambu class of brackets.

\section{Preliminaries} \label{SecPrelim}
\subsection{The generalized Nambu\/-\/determinant Poisson brackets}
In the context of quark dynamics and $n$-ary interactions, Nambu introduced (\cite{Nambu}, cf.~\cite{Donin1997,Grabowski93}) a class of Poisson brackets with global Casimirs $\ba=(a_1$,\ $\ldots$,\ $a_{d-2})$ on~$\BBR^d\ni\bx$: the Poisson bi\/-\/vectors are derived --\,w.r.t.\ the top\/-\/degree multivector $\varrho(\bx)\cdot\dd_{x^1}\wedge\ldots\wedge\dd_{x^d}$ on~$\BBR^d$\,-- using the Schouten bracket~$\lshad\cdot,\cdot\rshad$,
\[
P=\lshad\cdots\lshad \varrho(\bx)\,\dd_{x^1}\wedge\ldots\wedge\dd_{x^d}, a_1\rshad\ldots, a_{d-2}\rshad,
\]
with a not necessarily constant inverse of the volume density,~$\varrho(\bx)$. The coordinate expressions are, for example,
\[
\{f,g\} = \varrho(\bx)\cdot\left|\frac{\partial(a,f,g)}{\partial(x,y,z)}\right| =
  \varrho(x,y,z)\cdot \left| \begin{matrix} a_x & f_x & g_x \\ a_y &
f_y & g_y \\ a_z & f_z & g_z \end{matrix}\right|
\]
on~$\BBR^3\ni\bx=(x,y,z)$, and likewise,
\[
    \{f,g\} = \varrho(x,y,z,w)\cdot\left|\frac{\partial(a_1,a_2,f,g)}{\partial(x,y,z,w)}\right|
\]
on~$\BBR^4$ with global (e.g., Cartesian) coordinates $\bx=(x,y,z,w)$. It is obvious that the given functions $a_i$ which show up in the construction of the bi\/-\/vector $P$ Poisson\/-\/commute with any argument $f\in C^\infty(\BBR^d)$.
The scalar functions $a_i(\bx)=a_i(\bX'(\bx))$ do not change under the coordinate transformations $\bx(\bX')\rightleftarrows\bX'(\bx)$. Given two scalar functions $f,g\in C^\infty(\BBR^d)$, their Poisson bracket is also a scalar function. To counterbalance the behaviour of the Jacobian determinant in the course of coordinate transformations,
\[
   \left|\frac{\partial(a,f,g)}{\partial(x,y,z)}\right| = \left|\frac{\partial(a,f,g)}{\partial(x',y',z')}\right| \cdot \left|\frac{\partial(x',y',z')}{\partial(x,y,z)}\right|,
\]
the coefficient~$\varrho$ of top\/-\/degree multivector on~$\BBR^d$ behaves accordingly, $\varrho(\bx)\rightleftarrows\varrho'(\bX')$. E.g., on~$\BBR^3$ we have that
\[
   \varrho(x,y,z) \cdot \left|\frac{\partial(x',y',z')}{\partial(x,y,z)}\right| = \varrho'(x',y',z'),
\]
with an elementary general fact that $\ld\bx/\varrho(\bx)=\ld\bX'/\varrho'(\bX')$ and 
\begin{equation} \label{EqChangeRho}
\varrho(\bx)\cdot\left|
{\partial(\bX')}\big/{\partial\bx}\right|=\varrho'(\bX')
\end{equation}
for all dimensions $d\geqslant3$. So, let us keep in mind that the coefficient 
$\varrho(\bx)=\varrho'(\bX')\cdot\left|\partial\bx/\partial\bX'\right|$ of the top\/-\/degree multivector $\varrho(\bx)\cdot\dd_{x^1}\wedge\ldots\wedge\dd_{x^d}$
is nontrivially reparametrized under the changes $\bx(\bX')\rightleftarrows\bX'(\bx)$, whereas the scalar functions~$a_i$ are not transformed. Let us remember also that so far, the coordinate changes could be arbitrarily nonlinear, that is, not necessarily linear or affine on~$\BBR^d$.

\begin{lemma}\label{Lemma3DtrivialEvolve}
Let $a\colon\mathbb{R}^3 \to \mathbb{R}$ be a differentiable function, $\vec{Y} \in \Gamma(T\mathbb{R}^3)$ be a ($C^1$-)vector field on $\mathbb{R}^3$, and $T \in \Gamma(\wedge^3 T\mathbb{R}^3)$ be a (differentiable) tri\/-\/vector on~$\mathbb{R}^3$; refer it to any global (e.g., Cartesian) coordinates $x,y,z$ on $\mathbb{R}^3$ by the formula $T  =\varrho(x,y,z)\, \partial_x \wedge \partial_y \wedge \partial_z$.
For convenience, put $\vec{X} = -\vec{Y}$.
Then the scalar function~$a$ and tri\/-\/vector~$T$ evolve along the integral trajectories of the vector field~$\vec{Y}$ such that their Lie derivatives are, respectively,
\begin{subequations}
\label{EqEvolveARhoAlongVF}
\begin{align}
L_{\vec{Y}}(a) &= \schouten{\vec{Y}, a} = \vec{Y}(a) = -\vec{X}(a)\\
\intertext{and}
L_{\vec{Y}}(T) &= \schouten{\vec{Y}, T} = \schouten{T, \vec{X}},
\end{align}
\end{subequations}
so that $\left(\tfrac{\partial}{\partial t_Y}\varrho\right) \cdot \partial_x \wedge \partial_y \wedge \partial_z = \schouten{\varrho\, \partial_x \wedge \partial_y \wedge \partial_z, \vec{X}},$
where $t_Y$ is the natural parameter along an integral trajectory of $\vec{Y}$.\\
$\bullet$\quad Let $P$ be a Poisson bi-vector on $\mathbb{R}^3$.
Whenever the vector field~$\vec{Y}$ is shifted by a Hamiltonian vector field $\schouten{P, H}$, the evolutions of the scalar function~$a$ and of the coefficient~$\varrho$ in~$T$ respond by 
\begin{align*}
L_{\vec{Y} + \schouten{P,H}}(a) &= \vec{Y}(a) + \{a, H\}_P \\
\intertext{and }
L_{\vec{Y} + \schouten{P,H}}(T) &= \frac{\partial}{\partial t_Y} \varrho \cdot \partial_x \wedge \partial_y \wedge \partial_z + \partial_P(\schouten{T,H}),
\end{align*}
where $\partial_P = \schouten{P, \cdot}$ is the Poisson differential.\\
$\bullet$\quad Consider the generalized Nambu\/-\/determinant Poisson bi\/-\/vector $P_{\text{3D}} = \schouten{T, a}$ on~$\mathbb{R}^3$.
Trivial in its second Poisson cohomology, the evolution of this Poisson bi\/-\/vector,
\[
L_{\vec{Y}}(P) = \schouten{Y, P} = \partial_P(\vec{X}),
\] 
correlates with its evolution,
\[
\frac{\partial P}{\partial t_Y}([\varrho], [a]) = P(\frac{\partial \varrho}{\partial t_Y}, a) + P(\varrho, \frac{\partial a}{\partial t_Y}),
\] 
as soon as that is induced from evolution~\eqref{EqEvolveARhoAlongVF} of the Casimir~$a$ and coefficient~$\varrho$ of the tri\/-\/vector~$T$.
\end{lemma}

\begin{proof}
Formulas \eqref{EqEvolveARhoAlongVF} are standard in the calculus of multivectors.
The second 
claim and the correlation of evolutions amount to the Leibniz rule shape,
\begin{equation}\label{JacForSchoutenAsLeibniz}
\schouten{A, \schouten{B,Z}} = \schouten{\schouten{A,B}, Z} + (-)^{(|A|-1)(|B|-1)} \schouten{B, \schouten{ A, Z}}
\end{equation}
of the Jacobi identity for the Schouten bracket $\schouten{\cdot, \cdot}$ for (homogeneous) multivectors: e.g., take $A = \vec{X}, B = T$ and~$Z = a$.
\end{proof}

\begin{rem}
In the same way, now by repeated use of the Jacobi identity for $\schouten{\cdot, \cdot}$, one verifies that a Poisson\/-\/trivial evolution $Q = \schouten{P, \vec{X}}$ of the generalized Nambu\/-\/determinant Poisson bi\/-\/vector $P_{\text{4D}} = \schouten{ \schouten{K, a_1}, a_2}$ on $\mathbb{R}^4$ is correlated with the evolution of two Casimirs $a_1$ and $a_2$,
\begin{align*}
\tfrac{\partial}{\partial t} a_i &= -\vec{X}(a_i), \qquad i=1,2,\\
\intertext{and of the top-degree multivector $K = \varrho(x,y,z,w)\,\partial_x \wedge \partial_y \wedge \partial_z \wedge \partial_w$,}
\tfrac{\partial}{\partial t} \varrho &\cdot \partial_x \wedge \partial_y \wedge \partial_z \wedge \partial_w = \schouten{K, \vec{X}}.
\end{align*}
Indeed, we have that 
\begin{multline*}
\schouten{P_{\text{4D}}, \vec{X}} = \schouten{-\vec{X}, \schouten{\schouten{K,a_1}, a_2}} \\
{}= \schouten{ \schouten{ \schouten{-\vec{X}, K}, a_1 }, a_2} + \schouten{ \schouten{K, \schouten{-\vec{X}, a_1}}, a_2} + \schouten{ \schouten{K,a_1}, \schouten{-\vec{X}, a_2}}.
\end{multline*}
Whenever the vector field $\vec{Y} = -\vec{X}$ on~$\mathbb{R}^4$ is shifted by a Hamiltonian vector field $\schouten{P, H}$, the (signs in the) shifts of evolutions are obtained, along the above lines, for the Casimirs~$a_i$ and $4$-vector $\varrho\, \partial_x \wedge \partial_y \wedge \partial_z \wedge \partial_w$.
\end{rem}

\subsection{Kontsevich's graph flows} \label{SecGraphFlows}
In the seminal paper~\cite{Ascona96} (see also~\cite{Bourbaki2017} and~\cite{f16,sqs17,OrMorphism,Open2019,skew22,Banach19} for illustrations and discussion), Kontsevich designed a method to construct infinitesimal symmetries $P=Q([P])$ of the spaces of Poisson structures on affine finite\/-\/dimensional Poisson manifolds $(\mathit{M}^d_{\text{aff}},P)$. The construction is universal for all such geometries (with $\mathbf{x}'=A\bx+\boldb$ as the only admissible coordinate reparametrizations). The right\/-\/hand side $Q$ of the evolution $\dot{P}=Q([P])$, differential\/-\/polynomial in the components of the bi\/-\/vector $P$, is described by using linear combinations (with real coefficients) of directed graphs; these graphs are built of wedges ${\leftarrow}{\bullet}{\rightarrow}$ with prescribed ordering Left $\prec$ Right of the outgoing arrows 
in every internal vertex. Each edge is decorated with its own summation index which runs from 1 to the dimension $d=\text{dim }M^d$; each decorated edge $\xrightarrow{\phantom{ll}i\phantom{lll}}$ corresponds to the derivative $\partial/\partial x^i$ w.r.t.\ a local coordinate in an affine chart of~$\mathit{M}^d$; each internal vertex of the directed graph is inhabited by a copy of the Poisson bi\/-\/vector~$P=(P^{ij}(\bx))$. Each graph determines a differential\/-\/polynomial expression (w.r.t.\ the structure $P$ and the content of sink vertices) in a natural way: take the product of the (differentiated) contents of the
vertices and sum over all the indices. Two factors, namely (\textit{i}) the contraction of lower indices --\,from $\partial/\partial x^i$ and $\partial/\partial x^j$ on the respective Left and Right outgoing edges\,-- with the first and second indices $i,j$ in the skew-symmetric bi\/-\/vector components $P^{ij}(\bx)$ in the arrowtail vertex, and (\textit{ii}) the independence of the Jacobians $A$ (in the affine changes $\mathbf{x}'=A\bx+\boldb$) from a point of two charts' overlap, make the Kontsevich construction well defined for an arbitrary choice of local affine coordinates on~$(\mathit{M}^d_{\text{aff}},P)$.

The graph cocycles $\gamma$ on $n$ vertices and $2n-2$ edges in the Kontsevich undirected graph complex (see~\cite{Ascona96} as well as~\cite{OrMorphism,Banach19} and references therein), when directed (inheriting the edge ordering from $\gamma$) and evaluated at $n$ copies of a given Poisson bi\/-\/vector, yield a natural class of Kontsevich's  graph flows $\dot{P}=\Or(\gamma)(P^{\otimes^n})$ on the spaces of Poisson structures. Willwacher's construction of suitable graph cocycles $\gamma$ from the Grothendieck\/--\/Teichm\"uller Lie algebra $\mathfrak{grt}$ gives us the main sequence to work with: Kontsevich's tetrahedron $\gamma_3$ (which is the wheel graph with three spikes), the Kontsevich\/--\/Willwacher pentagon\/-\/wheel cocycle $\gamma_5$ (see~\cite{JPCS17,JNMP17} and~\cite{sqs17}), the heptagon-wheel cocycle $\gamma_7$ (see~\cite{JNMP17} and~\cite{OrMorphism}), etc., and their iterated commutators (always on $n$ vertices and $2n-2$ edges, for instance with 9 vertices and 16 edges in $[\gamma_3,\gamma_5]$). The construction of Lie brackets on the vector space of graphs with wedge ordering of edges is explained in~\cite{Ascona96} and~\cite{WillwacherGRT2010-2015,JNMP17}.

\begin{example}[\cite{Ascona96,Bourbaki2017} and~\cite{f16,Banach19}]\label{ExG3-Flow}
The tetrahedron $\gamma_3$, when oriented by the morphism $\Or$ to the balanced (by $8:24=1:3$) sum,
\[
\tfrac{1}{2}\bigl( \Gamma_1(\mathit{1},\mathit{2}) - \Gamma_1(\mathit{2},\mathit{1}) \bigr) 
   + 3\bigl(\Gamma_2'(\mathit{1},\mathit{2}) - \Gamma_2'(\mathit{2},\mathit{1}) \bigr),
\]
of two skew\/-\/symmetrized bi\/-\/vector graphs built of wedges (see Fig.~\ref{FigTetraFlow})
\begin{figure}[htb] 
\unitlength=1mm
\special{em:linewidth 0.4pt}
\linethickness{0.4pt}
\begin{picture}(30.00,30.00)
\put(20.00,30.00){\vector(-1,-2){7.67}}
\put(12.33,15.00){\vector(4,3){14.67}}
\put(20.00,30.1){\vector(-4,3){0}}
\put(27.00,26){\line(-5,3){7.00}}
\put(20.00,29.67){\vector(1,-4){4.67}}
\put(24.7,10.65){\vector(1,-2){5.00}}
\put(24.7,10.65){\vector(-1,-2){5.00}}
\put(12.33,15.00){\vector(3,-1){12.33}}
\put(27.00,26.00){\line(-1,-6){2.2}}
\put(24.85,12.5){\vector(-1,-4){0.2}}
\put(0.00,14.67){\makebox(0,0)[lc]{$\Gamma_1={}$}}
\put(16,22.33){\makebox(0,0)[rb]{\tiny$R$}}
\put(22.7,17.67){\makebox(0,0)[rb]{\tiny$L$}}
\put(22.67,28.33){\makebox(0,0)[lb]{\tiny$R$}}
\put(25.67,18.5){\makebox(0,0)[lc]{\tiny$L$}}
\put(19.5,20.00){\makebox(0,0)[rb]{\tiny$R$}}
\put(17.67,12.33){\makebox(0,0)[rt]{\tiny$L$}}
\put(21.33,5.00){\makebox(0,0)[rc]{\tiny$L$}}
\put(28.33,5.00){\makebox(0,0)[lc]{\tiny$R$}}
\end{picture}
\qquad
\unitlength=1mm
\special{em:linewidth 0.4pt}
\linethickness{0.4pt}
\begin{picture}(30.00,30.00)
\put(20.00,30.00){\vector(-1,-2){7.67}}
\put(27,26){\vector(-4,-3){14.5}}
\put(20.00,30.1){\vector(-4,3){0}}
\put(27.00,26){\line(-5,3){7.00}}
\put(20.00,29.67){\vector(1,-4){4.67}}
\put(24.7,10.65){\vector(0,-1){7.00}}
\put(12.33,15){\vector(0,-1){7.00}}
\put(12.33,15.00){\vector(3,-1){12.33}}
\put(27.00,26.00){\line(-1,-6){2.22}}
\put(27.00,26){\vector(1,4){0}}
\put(0.00,14.67){\makebox(0,0)[lc]{$\Gamma_2'={}$}}
\put(15.83,22.33){\makebox(0,0)[rb]{\tiny$k'$}}
\put(22.00,17.67){\makebox(0,0)[rb]{\tiny$\ell$}}
\put(22.67,29.33){\makebox(0,0)[lb]{\tiny$m'$}}
\put(26,18.33){\makebox(0,0)[lc]{\tiny$j$}}
\put(19.33,20.00){\makebox(0,0)[rb]{\tiny$\ell'$}}
\put(17.67,13.33){\makebox(0,0)[rt]{\tiny$k$}}
\put(11.33,10.00){\makebox(0,0)[rc]{\tiny$m$}}
\put(25.5,8.00){\makebox(0,0)[lc]{\tiny$i$}}
\end{picture}
\caption{The components of Kontsevich's tetrahedral flow 
$\dot{P}(\mathit{1},\mathit{2}) = \Gamma_1(\mathit{1},\mathit{2}) + 3\bigl(\Gamma_2'(\mathit{1},\mathit{2}) - \Gamma_2'(\mathit{2},\mathit{1}) \bigr)$
on the space of Poisson bi\/-\/vectors~$P$ on~$\BBR^d$ in any dimension~$d\geqslant3$.}\label{FigTetraFlow}
\end{figure}
now encodes the differential\/-\/polynomial velocity of Poisson bi\/-\/vectors:
\begin{multline*}
Q_{\text{tetra}}(P) = 
 1\cdot \biggl(\frac{\partial^{3}P^{ij}}{\partial x^{k} \partial x^{\ell}  \partial x^{m}} \frac{\partial P^{kk'} }{\partial x^{\ell'}} \frac{\partial P^{\ell\ell'}}{\partial x^{m'}} \frac{\partial P^{mm'}}{\partial x^{k'}}\biggr)\frac{\partial}{\partial x^{i}} \wedge \frac{\partial}{\partial x^{j}}   \\  
{} + 3 \cdot \biggl( \frac{\partial^{2} P^{ij}}{\partial x^{k} \partial x^{\ell}} \frac{\partial^{2} P^{km}}{\partial x^{k'} \partial x^{\ell'} }  \frac{\partial P^{k'\ell}}{\partial x^{m'}}   \frac{\partial P^{m'\ell'}}{\partial x^{j}}\biggr)\frac{\partial}{\partial x^{i}} \wedge \frac{\partial}{\partial x^{m}}.
\end{multline*}
Indeed, we place copies of a given bi\/-\/vector $P$ into the internal vertices, match their first and second indices with the summation indices that decorate the arrows (note that the ordering, available in the digraph encoding in \textit{loc.\ cit.}, is not everywhere displayed in Fig.~\ref{FigTetraFlow}, but it is easily retrieved from the differential\/-\/polynomial formula), and for all values of all the indices, we take the sum of products of all the differentiated contents of the vertices. It is clear that for an arbitrary affine Poisson manifold, the flow $\dot{P}=\Or(\gamma_3)(P^{\otimes^4})$ is coordinate-free.

Other examples of nonlinear proper ($\not\equiv0$ if $P$ is Poisson) Kontsevich's graph flows are constructed in~\cite{sqs17} for the pentagon\/-\/wheel cocycle~$\gamma_5$ and in~\cite{OrMorphism} for the heptagon\/-\/wheel cocycle~$\gamma_7$ (see also~\cite{Banach19}).
\end{example}

\section{Structural stability of the Nambu\/-\/determinant brackets under Kontsevich's flows} \label{SecStability}
\noindent The first main question which we explore in this note is how, in precisely which way the class~\eqref{EqBiVectNambu} 
of generalized $(\varrho\not\equiv1)$ Nambu\/-\/determinant bi\/-\/vectors $P(\varrho,[\ba])$ on~$\BBR^d$ is stable under Kontsevich's universal deformations $\dot{P}=Q_{\gamma}([P])$ given by the graph cocyles $\gamma$. The other question for us to explore is the combinatorial mechanism of this stability. So, let us first inspect how the infinitesimal symmetries $\dot{P}=\Or(\gamma)(P^{\otimes^{\#V(\gamma)}})$ of the --\,actually, unknown\,-- space of \textit{all} Poisson brackets $P$ on~$\BBR^d$ (where $\BBR^d$~is viewed as an affine manifold) restrict to the subspace of Nambu\/-\/determinant Poisson brackets.

Because the Nambu\/-\/determinant bi\/-\/vectors $P(\varrho,[\ba])= \lshad\cdots\lshad\varrho(\bx)\cdot\dd_{\bx},a_1\rshad\ldots,a_{d-2}\rshad$   
are \textit{linear} in both the inverse density~$\varrho$ and Casimirs~$\ba=(a_1$,\ $\ldots$,\ $a_{d-2})$, the class $\{P(\varrho,[\ba])\}$ is stable if 
there exist the velocities $\dot{\varrho}$ and $\dot{a}$ (depending on the point $(\varrho,\ba)$ in the functional parameter space) such that the Leibniz rule for the time derivative $\dd/\dd t$ is verified by
\begin{equation}\label{EqLeibnizARho}
\dot{P}(\varrho,[\ba])=P(\dot{\varrho},[\ba]) + \sum\nolimits_{i=1}^{d-2}
  P(\varrho,[a_1],\ldots,[\dot{a}_i],\ldots,[a_{d-2}]).
\end{equation}
In particular, the stability of the class is achieved if the evolution~$\dot{\varrho}$ and~$\dot{\ba}$ is differential\/-\/polynomial (of finite degrees and differential orders) in the parameters that evolve,
\begin{equation} \label{EqARhoEvolve}
 \dot{\varrho}=\dot{\varrho}([\varrho],[\ba]), \qquad \dot{\ba}=\dot{\ba}([\varrho],[\ba]).
\end{equation}
The construction of the Kontsevich flow $\dot{P}=\Or(\gamma)(P^{\otimes^n})$ from a graph cocycle $\gamma$ on $n$ vertices and the count of homogeneities always allow us to estimate both the order and polynomial degrees of such (non)linear PDE evolution~--- provided it exists. 

\begin{example}[$\gamma_3$-flow over~$\BBR^3$]\label{ExR3G3}
First, if $\varrho\equiv1$ and the Poisson bi\/-\/vector is $P=\lshad \dd_x\wedge\dd_y\wedge\dd_z,a(x,y,z)\rshad
$, then the Kontsevich tetrahedral flow $\dot{P}=\Or(\gamma_3)(P^{\otimes^4}[a])$ vanishes identically. In retrospect, this is true because every term in $\dot{a}$ contains a derivative of~$\varrho$, and all the more each term in $\dot{\varrho}$ does so, whence the Cauchy datum $\varrho=\text{const}$ makes the flow well defined but identically zero.

Let the inverse density $\varrho(x,y,z)$ be not necessarily constant over~$\BBR^3$. A simple \textit{a priori} estimate of homogeneities suggests that the terms in the differential\/-\/polynomial right\/-\/hand side of $\dot{\varrho}$ and $\dot{a}$ are constrained by the ansatz
\begin{align}\label{EqG3ARhoDotAnsatz}
\begin{split}
    \dot{a}\sim a^4\varrho^3 &\text{, with 9 derivatives in each monomial,} \\
    &\text{at most 3rd order derivatives of $a$ and of $\varrho$;} \\
    \dot{\varrho}\sim a^3\varrho^4 &\text{, with 9 derivatives in each monomial,} \\
    &\text{at most 3rd order derivatives of $a$ and of $\varrho$.} \\
\end{split}
\end{align}
By using the method of undetermined coefficients, implementing the problem in software for differential calculus on jet spaces (e.g., \textsf{Jets} by M.\,Marvan~\cite{Jets} or \textsf{gcaops} by R.\,Buring~\cite{BuringDisser}), we obtain the nontrivial solution (see also Example~\ref{ExShortCutARhoDot} on p.~\pageref{ExShortCutARhoDot} below). The differential polynomial $\dot{a}([\varrho],[a])$ consists of 228~monomials with nonzero coefficients, and $\dot{\varrho}([\varrho],[a])$ is 426~monomial long. It is seen that the actual dependence of~$\dot{a}$ and~$\dot{\varrho}$ on the jet variables $a_{\sigma}$ and $\varrho_{\tau}$ is such that the lengths of multi-indices $\sigma$ and $\tau$ are bounded by $1\leqslant|\sigma|\leqslant3$ and $0\leqslant|\tau|\leqslant1$ for $\dot{a}$ and by $1\leqslant|\sigma|\leqslant2$ and $0\leqslant|\tau|\leqslant3$ for~$\dot{\varrho}$. Apparent is also that in each monomial, there are exactly three derivatives w.r.t.\ $x$, exactly three w.r.t.~$y$, and exactly three w.r.t.~$z$. Here is a sample how these formulas read:
\begin{align*}
   \dot{a}&=-12\varrho^2a_x\varrho_ya_{xy}a_{zz}a_{xyz}+12\varrho^2a_x\varrho_ya_{xy}a_{xz}a_{yzz}+12\varrho^2a_x\varrho_ya_{xy}a_{xzz}a_{yz}+\ldots, \\
   \dot{\varrho}&=-12\varrho\varrho_x\varrho_ya_xa_z\varrho_{xxy}a_{zz}-12\varrho\varrho_x\varrho_ya_xa_z\varrho_{xzz}a_{yy}+24\varrho\varrho_x\varrho_ya_xa_z\varrho_{xyz}a_{yz}+\ldots;
\end{align*}
both formulas are given in full in Appendix~\ref{AppARhoDotG3}. In what follows, we shall explain these empiric facts; by understanding the combinatorics in these formulas, we collapse them to tiny Eqs.~\eqref{EqG3ARhoDotCollapsed} on p.~\pageref{EqG3ARhoDotCollapsed}.
\end{example}

\begin{example}[$\gamma_3$-flow over~$\BBR^4$]\label{ExG3R4}
In contrast with~3D, the tetrahedral flow $\dot{P}=\Or(\gamma_3)\linebreak[1](P^{\otimes^4})$ is nonzero for the ``authentic" Nambu\/-\/determinant Poisson bi\/-\/vector $P=\lshad\lshad \dd_x\wedge\dd_y\wedge\dd_z\wedge\dd_w,a_1\rshad,a_2\rshad
$ with pre\/-\/factor $\varrho\equiv1$ on~$\BBR^4\ni\bx=(x,y,z,w)$. With this Cauchy datum $\varrho\equiv1$ implying $\dot{\varrho}\equiv0$, we obtain that the differential polynomial velocities $\dot{a}_1([a_1],[a_2])$ and $\dot{a}_2([a_1],[a_2])$ each contain 9024 monomials (of two unequal differential profiles, 4512 and 4512~each, 
see Example~\ref{ExG3R4Rho1ThreeCivitas}).

Now the full case on~$\BBR^4$: take the generalized Nambu\/--\/Poisson bi\/-\/vector $P(\varrho,[a_1],[a_2])\linebreak[1]=
\lshad\lshad \varrho(\bx)\cdot\dd_x\wedge\dd_y\wedge\dd_z\wedge\dd_w,a_1\rshad,a_2\rshad
$. The tetrahedral flow $\dot{P}=\Or(\gamma_3)(P^{\otimes^4})$ does preserve this class of Poisson brackets: there exist differential\/-\/polynomial velocities of the inverse density $\varrho(\bx)$ and of the two Casimirs~$a_1$,\ $a_2$ 
such that\footnote{See the file \texttt{https://rburing.nl/gcaops/adot\_rhodot\_g3\_4D.txt}} 
\begin{align*}
    &\dot{\varrho}=\dot{\varrho}([\varrho],[a_1],[a_2]) \text{ with 90,024 terms,} \\
    &\dot{a}_1,\dot{a}_2([\varrho],[a_1],[a_2]) \text{ with 33,084 terms each.}
\end{align*}
The combinatorial structure of these right\/-\/hand sides in the general case ($\varrho\not\equiv1$) can be analysed by the technique which we develop in what follows: each of the three expressions is collapsed by using the marker\/-\/monomials for the triple summation with the Civita symbols on~$\BBR^4$. For instance, the differential monomials in either $\dot{a}_1$ or $\dot{a}_2$ are partitioned according to the homogeneity profiles of derivatives of $\varrho$, $a_1$, and $a_2$ with respect to the four coordinates on~$\BBR^4$ (see Table~\ref{TabG3R4GenRhoCount} on p.~\pageref{TabG3R4GenRhoCount} below). And all the 33,048 terms in $\dot{a}_1$ and $\dot{a}_2$ are expressed by formulas~\eqref{EqA1dotA2dotG3R4GenRho} on p.~\pageref{EqA1dotA2dotG3R4GenRho}.
\end{example}

\begin{example}[$\gamma_5$-flow over~$\BBR^3$]\label{ExG5R3}
The pentagon\/-\/wheel flow $P=\Or(\gamma_5)(P^{\otimes^6})$ on the space of all Poisson structures on~$\BBR^3$ restricts to the Nambu\/-\/determinant class of brackets $\{P(\varrho,[a])\}$. In the differential\/-\/polynomial formulas of evolution $\dot{\varrho}([\varrho],[a])$ and $\dot{a}([\varrho],[a])$, the right\/-\/hand side of $\dot{a}$ contains 79,212 monomials, and there are as many as 146,340 in $\dot{\varrho}$ (before either formula is collapsed by using five Civita symbols). 
Both these formulas of~$\dot{\varrho}$ and~$\dot{a}$ are stored externally in the file
\begin{quote}
\verb"https://rburing.nl/gcaops/adot_rhodot_g5_3D.txt"
\end{quote}
In the meantime, one can estimate the homogeneity degrees and orders, that is the polynomial degrees of each term in $\dot{a}$ and $\dot{\varrho}$ with respect to the jet variables $a_{\sigma}$ and $\varrho_{\tau}$, as well as the bounds on the possible (but not necessarily attained) lengths of the multi\/-\/indices $\sigma$ and $\tau$ counting the derivatives. We note that in every monomial, there are 5~subscripts $x$ (for derivatives, which is the usual notation), 5~subscripts $y$, and 5~subscripts $z$; both $\dot{a}$ and $\dot{\varrho}$ are manifestly skew\/-\/symmetric w.r.t.\ permutations of the base variables $x,y,z$ (meaning that the right\/-\/hand sides contain at least one Civita symbol~$\veps^{i_1i_2i_3}$).
\end{example}

\section{The structure of induced evolution $\dot{\varrho}$, $\dot{\ba}$} \label{SecHowEvolve}
\subsection{Encoding $\dot{\ba}$, $\dot{\varrho}$ by the Kontsevich graphs} \label{SecDotG}
The Kontsevich flow $\dot{P}=\Or(\gamma)\linebreak[1](P^{\otimes^n})$ on the class of generalized ($\varrho\not\equiv1$) Nambu\/-\/determinant Poisson brackets $P(\varrho,[\ba])$ preserves their structure. Let us interpret this fact back in the language of Kontsevich's directed graphs.

\begin{proposition}[\cite{MKConjecture13Dec2019adot}]\label{PropMKadotGraphs}
The evolution $\dot{a_i}$ of each Casimir in the Jacobian determinant within the Nambu\/--\/Poisson bracket, 
\begin{equation}
    \{f,h\}_{P(\varrho,[\ba])}=\varrho(\bx)\cdot\det\bigl\lVert\partial(a_1,\dots,a_{d-2},f,g)/\partial(x^1,\dots,x^d)\bigr\rVert, \tag{\ref{EqPBrNambu}}
\end{equation}
is equal to the value of the graph orientation morphism~$\Or$ at the $n$-\/tuple $P^{\otimes^{n-1}}\otimes a_i$ (here $n=\#\text{Vert}(\gamma)$):
\begin{equation} \label{EqSpeedCasimir}
    \dot{a_i}=\Or(\gamma)(P^{\otimes^{n-1}} \otimes a_i),
\end{equation}
where the right\/-\/hand side represents the sum of $n$-\/linear polydifferential operators which are encoded by the directed graph cocycle $\Or(\gamma)$ and which are evaluated at~$a_i$ placed consecutively in one of the vertices and the other vertices filled in by copies of the bi\/-\/vector~$P(\varrho,[\ba])$.
\end{proposition}

\begin{example}
For the tetrahedron~$\gamma_3$, which is symmetric w.r.t.\ each vertex of the full graph, we have that (let $d=3$ and denote the Casimir~$a_1$ by~$a$)
\begin{multline*}
\dot{a} = \Or(\gamma_3)(P\otimes P\otimes P\otimes a) \\
{}= \Or(\gamma_3)(a,P,P,P) + \Or(\gamma_3)(P,a,P,P) + \Or(\gamma_3)(P,P,a,P) + \Or(\gamma_3)(P,P,P,a) \\
{}= 4\,\Or(\gamma_3)(P,P,P,a),
\end{multline*}
as soon as the four vertices of~$\gamma_3$ are labelled and the list in the above formula, elements of which are separated by commas, indicates which argument is placed in which vertex of the graph.
\end{example}

\begin{proof}[Commentary]
Indeed, the Kontsevich graph flows $\dot{P}=\Or(\gamma)(P^{\otimes^n})$ are such that no arrows fall on the checked factors~$\check{\varrho}$ and~$\check{a}_i$ in the Leibniz formula for~$\dot{P}$, 
\begin{equation}\tag{\ref{EqLeibnizARho}} 
\dot{P}([\varrho],[\ba])= P(\dot{\varrho},[\check{\ba}]) + \sum\nolimits_{i=1}^{d-2}P(\check{\varrho},[\check{a}_1],\ldots, [\dot{a}_i],\ldots,[\check{a}_{d-2}]).
\end{equation} 
More specifically, to let exist the restriction of Kontsevich's graph flow to the class of Nambu\/-\/determinant Poisson 
bi\/-\/vectors~\eqref{EqBiVectNambu},
the directed graph formula, working over the content of each internal vertex by the Leibniz rule for each in-coming arrow, automatically singles out the terms in which (\textit{i}) the pre-factor $\varrho$ remains intact and (\textit{ii}) the in-coming derivatives are not spread over several Casimirs in the Jacobian inside that vertex. Nontrivial in this claim is that precisely all --\,without exception\,-- terms of such structure do form the well defined tuple of velocities~$\dot{\ba}$.
\end{proof}

\begin{cor} \label{CorRhoDotGraphs}
As soon as the evolution $\dot{\ba}$ of the Casimirs is obtained according to formula \eqref{EqSpeedCasimir}, 
from the structure~\eqref{EqPBrNambu} 
of the Nambu bracket and from the Leibniz rule in Eq.~\eqref{EqLeibnizARho} we deduce the speed of evolution for the inverse density~$\varrho$. Namely, we have that 
\begin{equation}\label{EqSpeedRho}
    \dot{\varrho}\cdot\left|\frac{\partial(a_1,\dots,a_{d-2},f,g)}{\partial(x^1,\dots,x^d)}\right|= 
    \bigg(\dot{P}([\varrho],[\ba])-\sum_{i=1}^{d-2}P(\varrho,[a_1],\dots,[\dot{a}_i],\dots,[a_{d-2}])\bigg)(f,g),
\end{equation}
where $f,g\in C^{\infty}(\BBR^d)$, the right\/-\/hand side with the known flow~$\dot{P}=\Or(\gamma)(P^{\otimes^n})$ is the value of the linear combination of bi\/-\/vectors at $f\otimes g$, and $\dot{\varrho}$ is extracted from the left\/-\/hand side by division.
\end{cor}

\begin{proof} [Commentary]
Indeed, by the above, the right\/-\/hand side is a whole multiple of the Jacobian determinant, which itself is equal to $P(\varrho\equiv1,[\ba])(f,g)$.
\end{proof}

\begin{example} \label{ExShortCutARhoDot}
The above proposition and corollary, resulting in the explicit differential\/-\/polynomial expressions for the velocities $\dot{\varrho}$ and $\dot{\ba}$ that induce a given graph flow $\dot{P}=\Or(\gamma)(P^{\otimes^n})$ for 
the Nambu structures $\lshad\cdots\lshad \varrho(\bx)\,\dd_{\bx},\cdots\ba\cdots\rshad\cdots\rshad$
on~$\BBR^d$, are illustrated 
in~\cite[Ch.\,6]{BuringDisser} by using the \textsf{gcaops} software for differential calculus on jet spaces. So far, the graph formulas are explicitly verified for
\begin{itemize}
\item the tetrahedral flow ($\gamma=\gamma_3$) on~$\BBR^3$;
\item the tetrahedral flow ($\gamma=\gamma_3$) on~$\BBR^4\ni\bx$ with $\varrho\equiv1$ (special case) and generic $\varrho(\bx)$ which implies $\dot{\varrho}\not\equiv0$;
\item the pentagon\/-\/wheel flow ($\gamma=\gamma_5$) on~$\BBR^3$ with generic $\varrho$.
\end{itemize}
For the tetrahedral $\gamma_3$-flow on~$\BBR^3$, the findings from Example~\ref{ExR3G3} are reproduced identically. A naive attempt to use the method of undetermined coefficients would be practically unfeasible in the other three cases, yet Eqs.~\eqref{EqSpeedCasimir} and~\eqref{EqSpeedRho} serve the correct formulas of~$\dot{\ba}$ and~$\dot{\varrho}$ without any need to solve a linear algebraic system.
\end{example}

\subsection{Civita symbols in $\dot{\varrho}$,\ $\dot{\ba}$: collapsing the formulas} \label{SecCivitas}
Our present task is to analyze the combinatorial structure of the differential\/-\/polynomial expressions for $\dot{\ba}$ and $\dot{\varrho}$ in formulas \eqref{EqSpeedCasimir} and \eqref{EqSpeedRho}, respectively, and collapse them as much as possible by using this new knowledge.

\subsubsection{The determinant provides one Civita symbol} \label{SecOneCivitaDet}
One simple fact is immediate from the presence of Jacobian determinant in the Nambu brackets.

\begin{proposition} \label{PropOnceSkew}
The differential polynomials $\dot{\ba}([\varrho],[\ba])$ and $\dot{\varrho}([\varrho],[\ba])$ are shifted skew\/-\/symmetric w.r.t.\ permutations of the base variables $x^1$,\ 
$\ldots$,\ $x^d$ (i.e.\ coordinates on the Poisson manifold~$\BBR^d$): for a graph cocycle $\gamma$ on $n$~vertices in each term, the velocities~$\dot{\varrho}$ and~$\dot{a_i}$ are skew\/-\/symmetric in $x^1$,\ $\ldots$,\ $x^d$ if $n$~is even (e.g., as for $\gamma_3$,\ $\gamma_5$,\ $\gamma_7$,\ $\ldots$,\ $\gamma_{2\ell+1}$,\ $\ldots$) and symmetric in $x^1$,\ $\ldots$,\ $x^d$ if $n$~is odd (e.g., such is the case for the cocycle $[\gamma_3,\gamma_5]$ on 9~vertices and 16~edges).
\end{proposition}

\begin{proof}
Every Poisson bracket of two scalar functions itself is a scalar function. For the Nambu bracket in particular,
\[
    \{f,g\}_{P(\varrho,[\ba])}=\varrho(\bx)\cdot\det\bigl\lVert\partial(a_1,\ldots,a_{d-2},f,g)/\partial(x^1,\ldots,x^d)\bigr\rVert,
\]
this invariance is provided by response~\eqref{EqChangeRho} of the inverse density $\varrho$ to a permutation $\sigma$ of rows in the Jacobian determinant: $\varrho(\bx)=(-)^{\sigma}\cdot\varrho'(\bX'(\bx))$ if $\bx=\sigma(\bX')$. A simple count shows that for a graph cocycle $\gamma$ on $n$ vertices (and $2n-2$ edges), 
in the respective differential polynomial of velocity~$\dot{\varrho}$ and~$\dot{\ba}$, each differential monomial contains the same number of jet variables for~$\varrho$ and~$a_i$, and the union of multi\/-\/indices from the derivatives in each monomial is always the same. Namely, we have that up to real coefficients,
\begin{align*}
    \dot{\varrho}&=\sum \Bigl( \text{terms} \sim \varrho^n\cdot a_1^{n-1}\cdot\ldots\cdot a_{d-2}^{n-1} &\text{with} (n-1)\times d \text{ base variables } x^1,\ldots,x^d \Bigr); \\
    \dot{a_i}&=\sum \Bigl( \text{terms} \sim \varrho^{n-1}\cdot a_i^n\cdot\mathop{{\prod}'}_{j\neq i}a_j^{n-1} &\text{with } (n-1)\times d \text{ base variables } x^1,\ldots,x^d \Bigr).
\end{align*}
For the velocities $\dot{a_i}$ to be scalars and for the objects $\dot{\varrho}$ to behave according to the same law, $\dot{\varrho}\bigr|_{\bx}=(-)^{\sigma}\dot{\varrho}'\bigr|_{\bX'(\bx)}$, as the inverse density $\varrho$ satisfies, both the right\/-\/hand sides have the parity $((-)^{\sigma})^{n-1}$ whenever the base variables are permuted: $\bx=\sigma(\bX')$.
\end{proof}

\begin{example}[$\gamma_3$-flow over~$\BBR^3$]\label{ExOnceSkewG3}
Indeed, for the tetrahedral $\gamma_3$-flow on the space of Nambu\/--\/Poisson structures $P(\varrho,[a])$ on~$\BBR^3$, with 228 terms in $\dot{a}$ and 426 terms in $\dot{\varrho}$, we verify that
\begin{align*}
    \dot{a}([\varrho],[a])(x,y,z)&=\sum\nolimits_{\sigma\in \BBS_3}(-)^{\sigma}\sigma(x,y,z)\text{ acts on (sum of $38$ terms)}, \\
    \dot{\varrho}([\varrho],[a])(x,y,z)&=\sum\nolimits_{\sigma\in \BBS_3}(-)^{\sigma}\sigma(x,y,z)\text{ acts on (sum of $71$ terms)}.
\end{align*}
The differential monomials in the right\/-\/hand sides are obtained by the greedy algorithm: for a monomial that still remains in the expression to be represented as an alternating sum, take its skew\/-\/symmetrization w.r.t.\ $\sigma\in \BBS_3$ acting on $x,y,z$, subtract it from the expression, collect similar terms and reduce, then proceed recursively until the list of monomials, initially met in the velocity, is empty.
\end{example}

We shall presently recognize such one\/-\/time skew\/-symmetrizations (when $n$ is even) within~$\dot{\varrho}$ and~$\dot{a_i}$ as a consequence of a much stronger claim about the independent action of $n-1$ copies of the permutation group~$\BBS_d$ on the 
$d$-\/tuples $\{x^1$,\ $\ldots$,\ $x^d\}_k$ for $1\leqslant k\leqslant n-1$ in the 
right\/-\/hand sides~$\dot{\varrho}$ and~$\dot{\ba}$. For instance, in the above example (here $n=4$ and $d=3$) the sign factor $(-)^{\sigma}$ is produced by the restriction on the diagonal, 
\[
\sum_{\sigma\in \BBS_3}(-)^{\sigma}\sigma(\bigotimes_{k=1}^{n-1}\{x,y,z\}_k)=\sum_{\sigma_1,\ldots,\sigma_{n-1}\in \BBS_3}(-)^{\sigma_1}\cdots(-)^{\sigma_{n-1}}\bigotimes_{k=1}^{n-1}\sigma_k(\{x,y,z\}_k)\Bigr|_{\sigma_k=\sigma},
\]
in the set of $n-1=3$ permutations $\sigma_k\in \BBS_d$ acting on the $n-1$ non-intersecting $d$-tuples $\{x^1$,\ $\ldots$,\ $x^d\}_k$ that partition the set of $(n-1)\times d$ derivatives occurring in the right\/-\/hand sides of $\dot{\varrho}$ and~$\dot{a}$.

\subsubsection{How $\varrho^k$ yields $k$ Civita symbols, or\textup{:} Jacobians generalized} \label{SecGenJacobians}
Let us recall three facts from analysis:
\begin{itemize}
\item the Casimirs $\ba=(a_1,\dots,a_{d-2})$ of the Nambu\/--\/Poisson brackets are scalar functions;
\item the inverse density~$\varrho$ obeys the transformation law $\varrho(\bx)=\varrho'(\bX')\cdot\det\lVert\partial\bx/\partial\bX'\rVert\bigr|_{\bX'(\bx)}$ under a change $\bx(\bX')\rightleftarrows\bX'(\bx)$;
\item the objects' velocities inherit the behaviour of those objects under 
coordinate transformations.
\end{itemize}
Consider a Kontsevich flow $\dot{P}=\Or(\gamma)(P^{\otimes^{n}})$ associated with a graph cocycle~$\gamma$ on~$n$ vertices. These three facts, put together, reveal that the reparametrization of derivatives of $\ba$ and $\varrho$ in the differential monomials within $\dot{\varrho}([\varrho],[\ba])$ and~$\dot{\ba}([\varrho],[\ba])$ match the nontrivial reparametrization of~$n-1$ copies of~$\varrho$ therein. 

More specifically, the $(n-1)\times d$ derivations,
namely $n-1$ copies of~$\dd_{x^i}$ for $1\leqslant i \leqslant d$,
arrange into $n-1$ totally skew\/-\/symmetric $d$-\/tuples 
$\veps^{i_1^1\dotsm i_d^1}\, \partial_{x^{i^1_1}}\otimes\dotsm\otimes \partial_{x^{i^1_d}}$,\ $\ldots$,\ $\veps^{i_1^{n-1}\dotsm i_d^{n-1}}\, \partial_{x^{i^{n-1}_1}}\otimes\dotsm\otimes \partial_{x^{i^{n-1}_d}}$, 
where $\veps^{\vec{\imath}^{\,\alpha}}$~is the Civita symbol on~$\BBR^d$. The derivatives from each $d$-tuple act on different comultiples of a \textit{marker\/-\/monomial} (which stands under the sum over the $n-1$ tuples $\vec{\imath}^{\,1},\ldots,\vec{\imath}^{\,\,n-1}$ with $d$ indices in each tuple and which thus marks, generally speaking, many differential monomials in the differential polynomial expressions $\dot{\varrho}$ and $\dot{a}_{\ell}$ when the sums over $\vec{\imath}^{\,\,\alpha}$ are expanded). In effect, each of the $d$-tuples $\partial_{x^1}\wedge\ldots\wedge\partial_{x^d}$ provides its own Jacobian determinant $\det\lVert\partial\bX'/\partial\bx\rVert$ when the coordinates are reparametrized on the affine base manifold~$\BBR^d$. These $n-1$ Jacobians $|\partial\bX'/\partial\bx|$ cancel against the $n-1$ Jacobians $|\partial\bx/\partial\bX'|$ from the reparametrizations of the inverse density $\varrho$ in either~$\dot{\varrho}$ or~$\dot{a}_\ell$. 

\begin{theor} \label{ThCivitas}
For a graph cocycle $\gamma$ on $n$ vertices, the Kontsevich flow $\dot{P}=\Or(\gamma)(P^{\otimes^n})$ restricts to the Nambu class~\eqref{EqPBrNambu} 
of Poisson brackets on the affine space~$\BBR^d$ in such a way that
\begin{align*}
    \dot{a}_{\ell} &= {}\\
 & \sum_{\sigma_1,\ldots,\sigma_{n-1}\in \BBS_d} \biggl(\prod_{i=1}^{n-1}(-)^{\sigma_i}\sigma_i((x^1,\ldots,x^d)_i)\biggr) \ (\text{marker\/-\/monomials} \sim\varrho^{n-1}a_{\ell}^n\cdot\mathop{{\prod}'}_{k\not=\ell} \lefteqn{a_k^{n-1}),}
\\
    \dot{\varrho} &= {}\\
 & \sum_{\sigma_1,\ldots,\sigma_{n-1}\in \BBS_d} \biggl(\prod_{i=1}^{n-1}(-)^{\sigma_i}\sigma_i((x^1,\ldots,x^d)_i)\biggr) \ (\text{marker\/-\/monomials} \sim\varrho^{n}a_1^{n-1}\dotsm \lefteqn{a_{d-2}^{n-1}).}
\end{align*}
The permutations $\sigma_i\in \BBS_d$ act on the partitioned set of subscripts 
$\bigsqcup\nolimits_{i=1}^{n-1}((x^1,\ldots,x^d)_i)$
\textup{(}for derivatives\textup{)} in each marker\/-\/monomial. Equivalently, we have that
\begin{align*}
    \dot{a}_{\ell} &= \sum_{\vec{\imath}^{\,1},\ldots,\vec{\imath}^{\,n-1}} \biggl(\bigotimes_{\alpha=1}^{n-1}\veps^{\vec{\imath}^{\,\alpha}} \cdot\partial_{\vec{\imath}^{\,\alpha}}\biggr) \:(\text{comultiples\:in\:marker\/-\/monomials} \sim\varrho^{n-1}a_{\ell}^n\cdot\mathop{{\prod}'}_{k\neq\ell} \lefteqn{a_k^{n-1} ),} \\
    \dot{\varrho} &= \sum_{\vec{\imath}^{\,1},\ldots,\vec{\imath}^{\,n-1}} \biggl(\bigotimes_{\alpha=1}^{n-1}\veps^{\vec{\imath}^{\,\alpha}} \cdot\partial_{\vec{\imath}^{\,\alpha}}\biggr) \:(\text{comultiples\:in\:marker\/-\/monomials} \sim\varrho^{n}a_{1}^{n-1}\cdot\ldots\cdot \lefteqn{a_{d-2}^{n-1} ),}
\end{align*}
with $d$-\/component multi\/-\/indices $\vec{\imath}^{\,\,\alpha}=(i_1^\alpha,\ldots,i_d^\alpha)$ in the Civita symbols $\veps^{\vec{\imath}^{\,\alpha}}$ on~$\BBR^d$.
\end{theor}

\begin{proof}[Commentary]
The right\/-\/hand side of the velocity $\dot{a}_\ell$ or $\dot{\varrho}$ is a very interesting object from analytic and combinatorial viewpoint: with all the first- and higher-order derivatives in the velocity, it looks like the Jacobian determinant w.r.t.\ \textit{each} tuple $(x^1$,\ $\ldots$,\ $x^d)\rightleftarrows(\partial_{x^1}$,\ $\ldots$,\ $\partial_{x^d})$, although the derivatives from different tuples can act on the same comultiple. It remains therefore to control the non\/-\/tensorial behaviour of the \textit{higher}-order derivatives (which do actually occur in the expressions under study, as seen from examples). Fortunately, this is where our initial assumption works: Kontsevich's graph flow is defined over an \textit{affine} manifold so that the second- and higher\/-\/order derivatives of the coordinate changes vanish identically for $\bx=A\mathbf{x}'+\boldb$ (with a constant Jacobian matrix~$A$). Thus, higher derivatives of~$a_\ell$ and~$\varrho$ are transformed by using only the first derivatives of the coordinate changes, whence the assertion.
\end{proof}

\begin{example}[$\gamma_3$-flow over~$\BBR^3$]\label{ExG3R3ThreeCivitas}
For the tetrahedral flow $\dot{P}=\Or(\gamma_3)(P^{\otimes^4})$ over~$\BBR^3$ we recall from Eq.~\eqref{EqG3ARhoDotAnsatz} in Example~\ref{ExR3G3} that
\begin{align*}
    \dot{a} &\sim \varrho^3a^4 \text{ with } xxxyyyzzz \text{ in each monomial,} \\
    \dot{\varrho} &\sim \varrho^4a^3 \text{ with } xxxyyyzzz \text{ in each monomial.}
\end{align*}
Now Theorem~\ref{ThCivitas} works: the $(4-1)\times3$ base variables are partitioned in three triples $(x,y,z)$ in each term, with a skew\/-\/symmetrization over each triple. Indeed, by a brute force calculation we verify that for the $\gamma_3$-flow over~$\BBR^3$,
\begin{equation}\label{EqG3ARhoDotCollapsed}
\begin{aligned}
    \dot{a}=\sum_{\sigma,\tau,\zeta\in \BBS_3} (-)^\sigma & (-)^\tau(-)^\zeta\bigl(
    2a_{u_1}a_{u_2}a_{u_3}\varrho_{w_1}\varrho_{w_2}\varrho_{w_3}a_{v_1v_2v_3} \\[-1em] 
    &{}-6\varrho a_{u_1v_2}a_{u_2}a_{u_3}\varrho_{w_1}\varrho_{w_3}a_{v_1v_3w_2} 
    -6\varrho^2a_{u_1}a_{u_2u_3}a_{v_1v_2}\varrho_{w_3}a_{v_3w_1w_2}\bigr), 
\\
    \dot{\varrho}=\sum_{\sigma,\tau,\zeta\in \BBS_3} (-)^\sigma & (-)^\tau(-)^\zeta
    \bigl(-2a_{u_1}a_{u_2}a_{u_3}\varrho_{v_1}\varrho_{v_2}\varrho_{v_3}\varrho_{w_1w_2w_3}
  \\[-1em]  
    &{}+6a_{u_1v_2}a_{u_2}a_{u_3}\varrho_{v_1}\varrho_{v_3}\varrho_{w_2}\varrho_{w_1w_3} 
    -12\varrho a_{u_1}a_{u_2u_3}a_{v_1v_2}\varrho_{v_3}\varrho_{w_1}\varrho_{w_2w_3}
  \\   
    &{}-6\varrho a_{u_1v_2}a_{u_2}a_{u_3}\varrho_{v_1}\varrho_{v_3}\varrho_{w_1w_2w_3} 
    +6\varrho^2 a_{u_1}a_{u_2u_3}a_{v_1v_2}\varrho_{v_3}\varrho_{w_1w_2w_3}\bigr),
\end{aligned}
\end{equation}
where each summation runs over three permutations $\sigma,\tau,\zeta \in \BBS_3$ giving three triples $(u_1,v_1,w_1) = (\sigma(x),\sigma(y),\sigma(z))$, also $(u_2,v_2,w_2) = (\tau(x), \tau(y), \tau(z))$, and $(u_3,v_3,w_3) = (\zeta(x), \zeta(y), \zeta(z))$.

We conclude that the 228 monomials in $\dot{a}$ and 426 monomials in $\dot{\varrho}$ which we started with are completely determined by only three marker\/-\/monomials for $\dot{a}$ and five marker\/-\/monomials for $\dot{\varrho}$ by using three Civita symbols in either 
formula.%
\footnote{\label{FootAllByEleven} Not only this: the three and five respective marker\/-\/monomials in both the velocities $\dot{a}$ and $\dot{\varrho}$ \textit{and} the 1,504 differential monomials in each component of the bi\/-\/vector flow $\dot{P}=\Or(\gamma_3)(P^{\otimes^4})$ for $P(\varrho,[a])$ on~$\BBR^3$ are completely determined by the eleven marker\/-\/monomials in the trivializing vector field $\smash{\vec{X}}$, which we obtain for the $\gamma_3$-flow over~$\BBR^3$ in section~\ref{SecG3R3VFCivitas}.} 
\end{example}

The natural question is \textit{how} the nine symbols $xxxyyyzzz$ in each term were distributed among the disjoint triples $xyz,xyz,xyz$ (to be permuted by $\sigma$, $\tau$ and $\zeta$ respectively); we shall analyze this in the next section.

\begin{example}[{$\gamma_3$-flow of $P(\varrho\equiv1,[a_1],[a_2])$ on~$\BBR^4$}]\label{ExG3R4Rho1ThreeCivitas}
Consider the ``authentic" Nambu\/-\/determinant bracket $P(\varrho\equiv1,[a_1],[a_2])$ and induce the $\gamma_3$-flow of the Casimirs~$a_1$ and~$a_2$, see Example~\ref{ExG3R4}. Owing to Theorem~\ref{ThCivitas} we collapse the 9,024 terms in either~$\dot{a}_1$ and~$\dot{a}_2$ to the thrice alternating formulas, namely
\begin{subequations}
\begin{align}
\dot{a}_1 = \sum_{\sigma,\tau,\zeta\in \BBS_4} (-)^\sigma(-)^\tau(-)^\zeta\bigl(
    &3a_{1;s_1u_2u_3}a_{1;t_1t_2}a_{2;s_2}a_{2;s_3v_1}a_{2;t_3u_1}a_{1;v_2}a_{1;v_3}
\notag\\[-0.5em]
    &{}+6a_{1;s_1u_2}a_{1;t_1}a_{1;t_2v_3}a_{1;u_3v_1v_2}a_{2;t_3u_1}a_{2;s_2}a_{2;s_3}
\bigr),\label{EqG3R4Rho1a1dot}
\\
\dot{a}_2 = \sum_{\sigma,\tau,\zeta\in \BBS_4} (-)^\sigma(-)^\tau(-)^\zeta\bigl(
    &3a_{1;s_1}a_{2;t_1u_2}a_{2;u_1u_3v_2}a_{1;s_2t_3}a_{1;s_3t_2}a_{2;v_1}a_{2;v_3}
\notag\\[-0.5em]
    &{}-3a_{1;s_1t_2}a_{2;u_1}a_{2;u_2u_3v_1}a_{1;t_1}a_{1;t_3}a_{2;s_2v_3}a_{2;s_3v_2}
\bigr),\label{EqG3R4Rho1a2dot}
\end{align}
\end{subequations}
where each summation runs over three permutations $\sigma,\tau,\zeta \in \BBS_4$ giving three $4$-\/tuples $(s_1,t_1,u_1,v_1) = (\sigma(x),\sigma(y),\sigma(z),\sigma(w))$, also $(s_2,t_2,u_2,v_2) = (\tau(x),\tau(y),\tau(z),\tau(w))$, and $(s_3,t_3,u_3,v_3) = (\zeta(x),\zeta(y),\zeta(z),\zeta(w))$.

Again, our task is to explain \textit{how} these formulas are obtained, i.e.\ how one can guess the right partitionings of $xxxyyyzzzwww$ in each monomial into three $4$-tuples $(x,y,z,w)$.
\end{example}

\begin{rem}\label{RemG3R4Rho1InequivMarkers}
The partitioning $xxxyyyzzzwww = xyzw \sqcup xyzw \sqcup xyzw$ within the second marker\/-\/monomial in the polynomial under the sum 
for the velocity~$\dot{a}_1$ in~\eqref{EqG3R4Rho1a1dot} is different from the analogous partitioning in the second marker\/-\/monomial (with coefficient $-3$) in the mirror\/-\/reflected formula~\eqref{EqG3R4Rho1a2dot} of the velocity~$\dot{a}_2$. The structural inequivalence of the two partitionings does occur modulo the relabelling $a_1\rightleftarrows a_2$ and modulo arbitrary reshuffles of the three $4$-\/tuples $\{s$,\ $t$,\ $u$,\ $v\}_k$ indexed by $k\in\{1,2,3\}$ and arbitrary permutations of $s$,\ $t$,\ $u$,\ $v$ in any of the $4$-\/tuples. Indeed, the last marker\/-\/monomial in~\eqref{EqG3R4Rho1a1dot} for~$\dot{a}_1$ contains the product of second 
derivatives $a_{1;s_1u_2}\cdot a_{1;t_2v_3}$ in which 
all the three $4$-\/tuples are mixed, whereas one of the $4$-\/tuples is \emph{not} present at all in the product of second derivatives $a_{2;s_2v_3}\cdot a_{2;s_3v_2}$ in the last marker\/-\/monomial in~\eqref{EqG3R4Rho1a2dot} for~$\dot{a}_2$.

Yet both the marker\/-\/monomials yield the sums (over $\sigma,\tau,\zeta\in \BBS_4$) which are mirror\/-\/reflections of each other under the swap $a_1\rightleftarrows a_2$. This is an example of marker\/-\/monomials' hidden symmetry which we discuss in the next section.
\end{rem}

\begin{example}[$\gamma_3$-flow over~$\BBR^4$ with $\varrho\not\equiv1$]\label{ExG3R4GenRhoThree}
The 33,084 terms in $\dot{a}_1$ or in its mirror-reflection $\dot{a}_2$ are captured  
--\,for the tetrahedral $\gamma_3$-flow on the space of generalized Nambu\/--\/Poisson brackets $P(\varrho,[a_1],[a_2])$ on~$\BBR^4\ni\bx=(x,y,z,w)$\,-- by three Civita symbols (or equivalently, by three permutations) using the formulas

{\tiny
\begin{subequations} \label{EqA1dotA2dotG3R4GenRho}
\begin{align}
    \dot{a}_1 =& \sum_{\sigma_1,\sigma_2,\sigma_3\in \BBS_4} (-)^{\sigma_1}(-)^{\sigma_2}(-)^{\sigma_3} \cdot \bigl(
3a_{1;s_1u_2u_3}a_{1;t_1t_2}a_{2;s_2}a_{2;s_3v_1}a_{2;t_3u_1}a_{1;v_2}a_{1;v_3}\varrho^3 
\notag\\
&+6a_{1;s_1u_2}a_{1;t_1}a_{1;t_2v_3}a_{1;u_3v_1v_2}a_{2;t_3u_1}a_{2;s_2}a_{2;s_3}\varrho^3 
+3a_{1;v_2}a_{1;t_1u_2v_3}a_{2;v_1}\varrho_{s_1}a_{1;u_1}a_{1;u_3}a_{2;s_2t_3}a_{2;s_3t_2}\varrho^2 
\notag\\
&-6a_{1;s_1v_3}a_{2;s_2t_1}\varrho_{v_1}a_{1;s_3u_1v_2}a_{1;u_2}a_{1;u_3}a_{2;t_2}a_{2;t_3}\varrho^2
-6a_{1;s_1v_2v_3}a_{1;t_1}a_{1;u_2}a_{1;u_3v_1}a_{2;s_2}a_{2;s_3u_1}a_{2;t_3}\varrho_{t_2}\varrho^2 
\notag\\ 
&+6a_{1;s_1}a_{1;s_2v_3}a_{1;s_3u_1}a_{1;t_1t_2t_3}a_{2;v_1}\varrho_{v_2}a_{2;u_2}a_{2;u_3}\varrho^2 
-6a_{1;s_1}a_{1;t_1u_2u_3}a_{2;u_1v_2}a_{1;t_2}a_{1;t_3}a_{2;s_2}a_{2;s_3}\varrho_{v_1}\varrho_{v_3}\varrho 
\notag\\
&+6a_{1;v_2}a_{1;s_1}a_{1;s_2s_3t_1}a_{1;t_2t_3}\varrho_{v_1}\varrho_{v_3}a_{2;u_1}a_{2;u_2}a_{2;u_3}\varrho
-2a_{1;s_1}a_{1;t_2}a_{1;t_3u_1u_2}a_{1;u_3}a_{2;t_1}a_{2;s_2}a_{2;s_3}\varrho_{v_1}\varrho_{v_2}\varrho_{v_3}
    \bigr), 
\\
    \dot{a}_2 =& \sum_{\sigma_1,\sigma_2,\sigma_3\in \BBS_4} (-)^{\sigma_1}(-)^{\sigma_2}(-)^{\sigma_3} \cdot \bigl(
3a_{1;s_1}a_{2;t_1u_2}a_{2;u_1u_3v_2}a_{1;s_2t_3}a_{1;s_3t_2}a_{2;v_1}a_{2;v_3}\varrho^3 
\notag\\ 
&-3a_{1;s_1t_2}a_{2;u_1}a_{2;u_2u_3v_1}a_{1;t_1}a_{1;t_3}a_{2;s_2v_3}a_{2;s_3v_2}\varrho^3 
-6a_{1;u_1}a_{2;t_1t_2v_3}\varrho_{t_3}a_{1;u_2v_1}a_{1;u_3v_2}a_{2;s_1}a_{2;s_2}a_{2;s_3}\varrho^2 
\notag\\ 
&+6a_{1;s_1}a_{1;t_1t_3}a_{1;u_2}a_{2;v_2}a_{2;s_2v_1v_3}a_{2;s_3t_2}a_{2;u_1}\varrho_{u_3}\varrho^2
+6a_{1;t_1u_2}a_{2;u_1v_2}\varrho_{u_3}a_{2;s_1s_2s_3}a_{1;v_1}a_{1;v_3}a_{2;t_2}a_{2;t_3}\varrho^2
\notag\\
&+3a_{2;v_1}a_{2;s_1s_2s_3}\varrho_{t_1}a_{2;t_2v_3}a_{2;t_3v_2}a_{1;u_1}a_{1;u_2}a_{1;u_3}\varrho^2
-6a_{1;t_1u_2}a_{2;u_1u_3v_2}a_{1;v_1}a_{1;v_3}\varrho_{t_2}\varrho_{t_3}a_{2;s_1}a_{2;s_2}a_{2;s_3}\varrho 
\notag\\
&-6a_{2;t_1u_2v_3}a_{2;u_1u_3}a_{2;v_1}a_{2;v_2}\varrho_{t_2}\varrho_{t_3}a_{1;s_1}a_{1;s_2}a_{1;s_3}\varrho
+2a_{2;s_1}a_{2;s_2s_3t_1}\varrho_{v_1}a_{2;v_2}a_{2;v_3}\varrho_{t_2}\varrho_{t_3}a_{1;u_1}a_{1;u_2}a_{1;u_3} 
    \bigr),
\end{align}
\end{subequations}

}

\noindent%
here $\{s_i,t_i,u_i,v_i\}=\sigma_i(x,y,z,w)$ for $\sigma_i\in \BBS_4$. Finding a compact
expression of $\dot{\varrho}\not\equiv0$ with 90,024 differential monomials in it, now using three Civita symbols, is a computationally much larger task than collapsing the velocities of the Casimirs.
\end{example}

\section{Marker\/-\/monomials and their hidden symmetry}\label{SecMarkers}
\noindent%
In this section we analyse \emph{how} the totally skew\/-\/symmetric differential polynomial velocities~$\dot{\varrho}$ and~$\dot{a}_\ell$ are represented by sums, with Civita symbols, using few differential monomials: we count their number and we explore the choice of suitable marker\/-\/monomials.

\begin{define} \label{DefMarkerMonomial}
A \textit{marker\/-\/monomial} in the fibre variables $\varrho$,\ $a_1$,\ $\ldots$,\ $a_{d-2}$ over the base variables $x^1$,\ $\ldots$,\ $x^d$ is a differential monomial in the jet variables $\varrho_{\kappa_\alpha}$,\ $a_{1;\lambda_{1,\beta}}$,\ $\ldots$,\ $a_{d-2;\lambda_{d-2,\delta}}$ (here the multi\/-\/indices $\kappa_\alpha,\lambda_{i,\beta}$ for the derivatives satisfy $0\leqslant|\kappa_\alpha|,|\lambda_{i,\beta}|<\infty$) such that $\sum_\alpha |\kappa_\alpha|+\sum_{i=1}^{d-2} \sum_\beta |\lambda_{i,\beta}|=\mu\cdot d$ with $\mu\in\BBN_{\geqslant1}$, such that $\bigcup_\alpha\kappa_\alpha \cup \bigcup_{i=1}^{d-2} \bigcup_\beta \lambda_{i,\beta} =\bigcup_{k=1}^{\mu}\{x^1$,\ $\ldots$,\ $x^d\}_k$, and such that all the base variables~$x^\ell$ in the multi\/-\/indices (denoting the respective derivatives) are partitioned into $\mu$~disjoint $d$-\/tuples $x^1$,\ $\ldots$,\ $x^d$. 

Every such tuple then corresponds to its own alternating sum $\sum_{\sigma\in \BBS_d}(-)^{\sigma}\sigma(x^1$,\ $\ldots$,\ $x^d)$:
each (of the $\mu$ in total) permutation~$\sigma_k$ acts on the respective $d$-\/tuple $\{x^1$,\ $\ldots$,\ $x^d\}_k$ 
of base variables within the multi\/-\/indices $\kappa_\alpha$,\ $\lambda_{i,\beta}$ of jet variables $\varrho_{\kappa_\alpha}$,\ $a_{1;\lambda_{1,\beta}}$,\ $\ldots$,\ $a_{d-2;\lambda_{d-2,\delta}}$ in the marker\/-\/monomial. 
Equivalently, $k$th sum over the permutation group $\BBS_d\ni\sigma_k$ corresponds to the Civita summation $\sum_{\vec{\imath}}\veps^{\vec{\imath}}$ chosen such that the marker\/-\/monomial ``as is" occurs with the plus sign when $\vec{\imath}=(1$,\ $2$,\ $\ldots$,\ $d)$, so the base variables $x^1$,\ $\ldots$,\ $x^d$ in the $d$-\/tuple are then represented by the variables $x^{i_1}$,\ $\ldots$,\ $x^{i_d}$ in the subscripts, respectively.
\end{define}

\begin{define}\label{DefZeroMarker}
A marker\/-\/monomial is called \textit{zero} if the alternating sum over all permutations of all the $d$-\/tuples $(x^1$, $\ldots$,\ $x^d)$ in it is identically equal to zero.
\end{define}

\begin{example} \label{ExMarkerMonomialsD2}
Let $x,y$ be the base variables and $\varrho$ be the fibre variable. Consider the marker\/-\/monomial $M_1=\varrho_{x_1}\varrho_{y_1}\varrho_{x_2y_2}$ with the two\/-\/tuples $\{x_1y_1\}\bigsqcup\{x_2y_2\}$. Taking the alternating sum,
\[
    \sum_{\sigma\in \BBS_2}\sum_{\tau\in \BBS_2} (-)^\sigma(-)^\tau \varrho_{\sigma(x)}\varrho_{\sigma(y)}\cdot\varrho_{\tau(x)\tau(y)} =(\varrho_x\varrho_y-\varrho_y\varrho_x) \cdot(\varrho_{xy}-\varrho_{yx})\equiv0,
\]
we establish that the marker\/-\/monomial $M_1$ is a zero marker.

\noindent$\bullet$\quad But let us instead take the marker\/-\/monomial $M_2=\varrho_{x_1}\varrho_{y_2}\varrho_{x_2y_1}$ with a different partitioning of the letters $xxyy$ as they are seen in $M_1$. We now have that
\begin{equation}
    \varrho_x\varrho_y\varrho_{xy}-\varrho_y\varrho_y\varrho_{xx}-\varrho_x\varrho_x\varrho_{yy}+\varrho_y\varrho_x\varrho_{yx}\not\equiv0.
\end{equation}
In other words, the new marker\/-\/monomial $M_2$ is not zero any longer, even though the profile $|\sigma_1|=1=|\sigma_2|$, $|\sigma_3|=2$ of the comultiples in the product $\varrho_{\sigma_1}\cdot\varrho_{\sigma_2}\cdot\varrho_{\sigma_3}$ is the same as in~$M_1$.
\end{example}

\begin{define} \label{DefDiffProfile}
The \textit{differential profile} of orders of the derivatives in a marker\/-\/monomial $M=\varrho_{\kappa_1}\varrho_{\kappa_2}\dots a_{1;\lambda_1}\dots a_{d-2;\mu_1}\dots$ is the set of pairs $\{\varrho|\kappa_1|$,\ $\varrho|\kappa_2|$,\ $\ldots$,\ $a_1|\lambda_1|$,\ $\ldots$,\ $a_{d-2}|\mu_1|$,\ 
$\ldots\} \mathrel{\stackrel{\text{def}}{=}} \{\varrho|\kappa_1||\kappa_2|\dots$,\ $a_1|\lambda_1|\dots$, $\ldots$,\ $a_{d-2}|\mu_1|\dots\}$: each (instance of a) fibre variable is followed by the nonnegative order(s) of its derivative(s).%
\footnote{\label{FootDiffProfileOrderTen}%
The first variant of notation is inevitable if some of the orders is at least~$10$; in this note, the other variant of notation is enough (see Tables~\ref{TabG3R3CountProfiles}--\ref{TabG3R4GenRhoCount} in the next section).}
\end{define}

\begin{example} \label{ExProfiles}
Both marker\/-\/monomials in Example~\ref{ExMarkerMonomialsD2} have the same differential profile $\varrho1\varrho1\varrho2$ (equivalently, $\varrho112$), yet $M_1$~is zero whereas $M_2$~is not zero as a marker.
\end{example}

The differential profile of a marker\/-\/monomial is thus a coarse invariant (w.r.t.\ permutations of all the base variable in it, or w.r.t.\ a permutation of the base variables within one of the $d$-tuples $x^1$,\ $\ldots$,\ $x^d$). It is clear also that marker\/-\/monomials of unequal differential profiles cannot be obtained one from another by permuting the comultiples or by permuting the base variables (what the alternating sum does by definition). This implies that to represent a differential polynomial by an alternating sum over the permutations which act on the base variable in the marker\/-\/monomials, the sums of terms of unequal differential profiles can be processed independently one from another. 

\begin{rem} \label{RemMarkerNotUnique}
Representations of a differential polynomial by using 
marker\/-\/monomials are not unique. Indeed, the marker can be picked for any value of the permutation(s). For instance, we have that 
\begin{multline*}
    \sum_{\sigma,\tau\in \BBS_2}(-)^\sigma(-)^\tau\varrho_{\sigma(x)}\varrho_{\tau(x)}\varrho_{\sigma(y)\tau(y)} = \sum_{\sigma,\tau\in \BBS_2}(-)^\sigma(-)^\tau\varrho_{\sigma(y)}\varrho_{\tau(y)}\varrho_{\sigma(x)\tau(x)} = {}\\
   {} = -\sum_{\sigma,\tau\in \BBS_2}(-)^\sigma(-)^\tau \varrho_{\sigma(x)}\varrho_{\tau(y)}\varrho_{\sigma(y)\tau(x)} = -\sum_{\sigma,\tau\in \BBS_2}(-)^\sigma(-)^\tau \varrho_{\sigma(y)}\varrho_{\tau(x)}\varrho_{\sigma(x)\tau(y)}.
\end{multline*}
Indeed, each of the four choices of the monomial marks the same expression, $\varrho_x^2\varrho_{yy}-2\varrho_x\varrho_y\varrho_{xy}+\varrho_y^2\varrho_{xx}\not\equiv0$.
\end{rem}

At the same time, for two nonzero marker\/-\/monomials of equal differential profiles it can be that their alternating sums are neither equal nor proportional to each other but intersect, that is, the two resulting differential polynomials have common term(s).

\begin{counterexample} \label{CounterExRightChoice} 
The monomial $a_x\varrho_xa_{yy}\varrho_{xy}$ is a term in the alternating sums for the markers
\[
    M_3 = a_{\sigma(x)}\varrho_{\tau(x)}a_{\sigma(y)\zeta(y)}\varrho_{\zeta(x)\tau(y)} \quad\text{and}\quad
    M_4 = a_{\sigma(x)}\varrho_{\tau(x)}a_{\tau(y)\zeta(y)}\varrho_{\zeta(x)\sigma(y)},
\]
indeed showing up when $\sigma=\tau=\zeta=\id$, but the two fully alternating sums are not equal,
\[
    \sum_{\sigma,\tau,\zeta\in \BBS_2}(-)^{\sigma}(-)^{\tau}(-)^{\zeta}\sigma\otimes\tau\otimes\zeta(M_3)\not=\sum_{\sigma,\tau,\zeta\in \BBS_2}(-)^{\sigma}(-)^{\tau}(-)^{\zeta}\sigma\otimes\tau\otimes\zeta(M_4),
\]
which can be seen by straightforward expansion. The two differential polynomials are not even multiples of one another.
\end{counterexample}

This implies that to represent a given sum, the marker\/-\/monomial can be unique (up to a given permutation of the base variables in a $d$-tuple) but the choice of the base variables' partitioning (into the disjoint $d$-tuples) can be not unique, and only the right choice does the job. This ambiguity yields a nontrivial problem of finding the ``true" partitioning of the $\mu\cdot d$ derivatives into $\mu$~tuples $\{x^1$,\ $\ldots$,\ $x^d\}$ in each term of the right\/-\/hand sides $\dot{\varrho}([\varrho],[\ba])$,\ $\dot{a}_\ell([\varrho],[\ba])$ for a given Kontsevich's graph flow on the space of Nambu\/--\/Poisson brackets~$P(\varrho,[\ba])$.

We discover that this anticipated ambiguity is heavily suppressed by an extra, so far hidden symmetry of these graph flows on this particular class of Poisson brackets on~$\BBR^d$.

\begin{proposition}[$\dot{a},\dot{\varrho}$ for $\gamma_3$-flow over~$\BBR^3$]\label{PropHyperSymG3R3}
In the evolution $\dot{\varrho},\dot{a}$ which is induced by the tetrahedral flow on the class of generalized ($\varrho\not\equiv1$) Nambu\/--\/Poisson brackets $P(\varrho,[a])$ on~$\BBR^3$, the count of differential monomials of unequal differential profiles is presented in Table~\ref{TabG3R3CountProfiles}.%
\footnote{\label{FootRedirectArrowsG3R3} From Proposition~\ref{PropMKadotGraphs} we recall that the velocity $\dot{a}$ is encoded using the Kontsevich graphs by formula~\eqref{EqSpeedCasimir}. Because the entire flow $\dot{P}=\Or(\gamma_3)(P^{\otimes^4})$ is specified by the directed graph cocycle $\Or(\gamma_3)$, the velocity $\dot{\varrho}$ is deduced from Eq.~\eqref{EqSpeedRho}. One can inspect in full detail how the arrows, targeted on a copy of~$a$ in the construction of~$\dot{a}$, spread over copies of~$\varrho$ and~$a$ to form~$\dot{\varrho}$ in Eq.~\eqref{EqLeibnizARho}. This is why there is much similarity in the differential profiles of terms in the two velocities (as seen from Table~\ref{TabG3R3CountProfiles}).}%
\begin{table}[htb]
\caption{\label{TabG3R3CountProfiles} The number of monomials and their differential profiles in~$\dot{a}$ and~$\dot{\varrho}$ for the tetrahedral $\gamma_3$-flow over~$\BBR^3$.}
\centerline{
    \begin{tabular}{|p{4cm}|p{4cm}|}
    \hline
    In $\dot{a}$ & In $\dot{\varrho}$ \\ 
    \hline
    54:\,\,\, $a1113\varrho111$ & 54: \,\,\,$a111\varrho1113$ \\
    \hline
    & 102: $a112\varrho1112$ \\
    \hline
    102: $a1123\varrho011$ & 102: $a112\varrho0113$ \\
    \hline
    & 96: \,\,\,$a122\varrho0112$ \\
    \hline
    72:\,\,\, $a1223\varrho001$ & 72: \,\,\,$a122\varrho0013$ \\
    \hline
    \end{tabular}}
\end{table}

\noindent$\bullet$\quad For each of the three differential profiles of monomials in~$\dot{a}$ and five in~$\dot{\varrho}$, we discover that for \textit{any} choice of nonzero marker\/-\/monomial with that profile, its total skew\/-\/symmetrization (using three permutations, each acting on its own tuple $xyz$), taken with a suitable 
nonzero coefficient, exactly equals the entire sum of all the terms with that differential profile. In other words, for each of the $3+5$ differential profiles of monomials in~$\dot{a}$ and~$\dot{\varrho}$ respectively, the total skew\/-\/symmetrizations of all nonzero markers of a fixed profile are multiples of each other.
\end{proposition}

This reveals a previously hidden, extra symmetry of the objects in Kontsevich's flow under study.

The case of Nambu\/--\/Poisson structures (with $\varrho\not\equiv1$) on~$\BBR^4$, when the tetrahedral $\gamma_3$-flow induces the evolution $\dot{a}_1,\dot{a}_2$ and $\dot{\varrho}\not\equiv0$, is even more interesting: we observe the exact same extra symmetry for all but one differential profiles, and one profile exceptionally requires the use of two marker\/-\/monomials.

\begin{proposition}[$\dot{a}_1$,\ $\dot{a}_2$ for $\gamma_3$-flow with $\varrho\not\equiv1$ over~$\BBR^4$]\label{PropHyperSymG3R4A1A2}
The count of differential monomials of unequal profiles in the velocities $\dot{a}_1$ and~$\dot{a}_2$ (see Example~\ref{ExG3R4Rho1ThreeCivitas}) is summarized in Table~\ref{TabG3R4GenRhoCount}.
(The symmetry in how the Casimirs $a_1$ and $a_2$ appear in the Nambu\/-\/determinant Poisson bracket is naturally reflected in their evolution under the tetrahedral $\gamma_3$-flow).
\begin{table}[htb]
\caption{\label{TabG3R4GenRhoCount} The count of monomials w.r.t.\ their differential profiles in~$\dot{a}_1$ and~$\dot{a}_2$ for the tetrahedral $\gamma_3$-flow on the space of generalized ($\varrho\not\equiv1$) Nambu\/--\/Poisson brackets on~$\BBR^4$.}
\centerline{
    \begin{tabular}{|p{5cm}|p{5cm}|} 
    \hline
    In $\dot{a}_1$ & In $\dot{a}_2$ \\ 
    \hline
    4512: $a_1 1123 a_2 122\varrho000$ & 4512: $a_1 122 a_2 1123\varrho000$ \\
    \hline
    4512: $a_1 1223 a_2 112\varrho000$ & 4512: $a_1 112 a_2 1223\varrho000$ \\
    \hline
    3168: $a_1 1113 a_2 122\varrho001$ & 3168: $a_1 122 a_2 1113\varrho001$ \\
    \hline
    7872: $a_1 1123 a_2 112\varrho001$ & 7872: $a_1 112 a_2 1123\varrho001$ \\
    \hline
    3168: $a_1 1223 a_2 111\varrho001$ & 3168: $a_1 111 a_2 1223\varrho001$ \\
    \hline
    3984: $a_1 1113 a_2 112\varrho011$ & 3984: $a_1 112 a_2 1113\varrho011$ \\
    \hline
    3984: $a_1 1123 a_2 111\varrho011$ & 3984: $a_1 111 a_2 1123\varrho011$ \\
    \hline
    1848: $a_1 1113 a_2 111\varrho111$ & 1848: $a_1 111 a_2 1113\varrho111$ \\
    \hline
    \end{tabular}}
\end{table}

\noindent$\bullet$\quad The homogeneous differential polynomial components of all profiles except the 7872 terms with $a_1 1123 a_2 112 \varrho001$ and the 7872 terms with $a_1 112 a_2 1123 \varrho001$ enjoy the same extra symmetry as at $d=3$: just one, arbitrarily chosen nonzero marker\/-\/monomial suffices to express the entire sum. In particular, this is always so in the restricted case $\varrho\equiv1$ when $\dot{\varrho}\equiv0$ and the nontrivial velocities $\dot{a}_1$,\ $\dot{a}_2$ realize the entire evolution of the class $\{P(\varrho\equiv1,[a_1],[a_2])\}$.

In either of the two exceptional cases (one in~$\dot{a}_1$ and the other in~$\dot{a}_2$, with necessarily $\varrho\not\equiv1$), when two marker\/-\/monomials are needed, the first choice is still arbitrary but the next choice is constrained by the former.%
\footnote{\label{FootLooksLikeRoots} This looks similar to the construction of a basis in $\BBE^2$ by using a root system with the Coxeter graph ${\bullet}\!\text{\textbf{---}}\!{\bullet}$: selecting the first vector is free but as one proceeds, the remaining direction is constrained.}
\end{proposition}

The marker\/-\/monomial expression of~$\dot{\varrho}$ in the generic case $\varrho\not\equiv1$ on~$\BBR^4$ carrying the tetrahedral $\gamma_3$-flow --\,and a simultaneous study of the presence or absence of the new extra symmetry in it\,-- is a computationally challenging problem; the same applies to the pentagon\/-\/wheel $\gamma_5$-flow on~$\BBR^3$ (to collapse the known evolution $\dot{a}$,\ $\dot{\varrho}\not\equiv0$ by using five Civita symbols and to check the extra symmetry in the course of building the hypotheses about the $\mu\cdot d=5\cdot3$ base variables' partitioning into~$\mu\cdot\{xyz\}$).

\section{Vector field~$\smash{\vec{X}}$ which trivializes the $\gamma_3$-\/flow of Nambu brackets}\label{SecVFTrivialFlows} \noindent
Finally, we examine the Poisson triviality of the \textit{restriction} of Kontsevich's graph flow $\dot{P}=\Or(\gamma_3)(P^{\otimes^n})$ to the space of Nambu\/--\/Poisson structures $P(\varrho,[\ba])$ on~$\BBR^d$. (There is no known mechanism for Kontsevich's graph flows to be trivial in the second Poisson cohomology of~$P$ for nontrivial graph cocycles~$\gamma$ and generic Poisson structures.)

\subsection{The trivializing vector field $\smash{\vec{X}}(\gamma_3,\varrho,a)$ and Civita symbols in it} \label{SecG3R3VFCivitas}
Let us make a few estimates of differential polynomial degrees and orders. For every graph cocycle $\gamma=\sum_{\ell}c_\ell\cdot\gamma_\ell$ with graphs $\gamma_\ell$ on $n$ vertices and $2n-2$ edges, the restriction of Kontsevich's flow $\dot{P}=\Or(\gamma)(P^{\otimes^n})$ to the space of generalized ($\varrho\not\equiv1$) Nambu\/-\/determinant Poisson bi\/-\/vectors $P(\varrho,[\ba])$ with $d-2$ global Casimirs $\ba=(a_1,\ldots,a_{d-2})$ on~$\BBR^d$ contains, in each term of the differential\/-\/polynomial coefficient of the bi\/-\/vector $\dot{P}(\varrho,[\ba])$, $n\cdot(d-2)+2n-2=nd-2$ derivatives spread over $\varrho^n\cdot a_1^{n}\dots a_{d-2}^{n}$. (Hence there are $(nd-2)-(d-2)=(n-1)d$ derivatives spread over $\varrho^{n-1}\cdot a_k^n\cdot  
\prod^{\prime}_{j\neq k}a_j^{n-1}$ in $\dot{a}_k$ and over $\varrho^n\cdot a_1^{n-1}\dotsm a_{d-2}^{n-1}$ in~$\dot{\varrho}$.) The trivializing vector field $\smash{\vec{X}}=\sum_{i=1}^dX^i([\varrho],[\ba])\,\partial/\partial x^i$ with differential\/-\/polynomial coefficients satisfying the coboundary equation $\Or(\gamma)(P^{\otimes^n})=\schouten{P,\smash{\vec{X}}}$ for $P(\varrho,[\ba])$ would therefore have $(nd-2)-(d-2)-1=(n-1)d-1$ derivatives spread over $\varrho^{n-1}\cdot a_1^{n-1}\dotsm a_{d-2}^{n-1}$ in every term of each coefficient $X^i$. The Civita mechanism of base coordinates' partitioning now applies to the trivializing vector field. For the object $\smash{\vec{X}}$ to be a vector field under coordinate reparametrizations $\bx(\bX')\rightleftarrows\bX'(\bx)$, the behaviour of $n-1$ comultiples~$\varrho$ dictates that there are $n-1$ Civita symbols in each~$X^i$:
\begin{equation} \label{EqVFCivitas}
    \smash{\vec{X}}=\sum\nolimits_{\vec{\imath}^{\,1},\ldots,\vec{\imath}^{\,n-1}} \veps^{\vec{\imath}^{\,1}}\dotsm\veps^{\vec{\imath}^{\,n-1}}\cdot X_{\vec{\imath}^{\,1},\ldots,\vec{\imath}^{\,n-2};i_1^{n-1}\dotsm i_{d-1}^{n-1}} \cdot\partial/\partial x^{i_d^{n-1}}.
\end{equation}
In other words, the vector field coefficients~$X^i$ collapse by using all the indices of $n-2$ Civita symbols $\veps^{\vec{\imath}^{\,\alpha}}$ on~$\BBR^d$ and by using all but one last index of the $(n-1)\text{th}$ Civita symbol.

By Lemma~\ref{Lemma3DtrivialEvolve},
if the trivializing vector field~$\vec{Y}=-\vec{X}$ exists
such that $\dot{P}\bigl([\varrho]$,$[\ba]\bigr) = \lshad P,\vec{X} \rshad$,
the velocities of scalar Casimirs are $\dd/\dd t_{Y}
({a}_k)=-\smash{\vec{X}}(a_k)$. 
Nontrivial here is that \textit{zero} marker\/-\/monomials can be produced in the velocity $
-\bigl(\sum_{i=1}^dX^i([\varrho],[\ba])\,\partial/\partial x^i\bigl)(a_k)$ from nonzero marker\/-\/monomials in the right\/-\/hand side of~\eqref{EqVFCivitas}. This prompts
that the velocity $\dot{a}_k$, which was obtained directly from the graph cocycle $\gamma$ by using formula~\eqref{EqSpeedCasimir}, can involve fewer marker\/-\/monomials than there are terms to express the coefficient~$X^i$ by~\eqref{EqVFCivitas}. We observe this effect already in the simplest case, namely for the Kontsevich tetrahedral flow (so $n=4$) and the generalized ($\varrho\not\equiv1$) Nambu\/-\/determinant Poisson structures $P(\varrho,[a])$ on~$\BBR^3$ (so $d=3$).
The same concerns~$\dot{\varrho}$ with only five terms in~\eqref{EqG3ARhoDotCollapsed} on p.~\pageref{EqG3ARhoDotCollapsed}.

\begin{theor} \label{ThG3R3VF}
The Kontsevich tetrahedral flow $\dot{P}=\Or(\gamma_3)(P^{\otimes^4})$ for the Nambu\/--\/Poisson brackets $P(\varrho,[a])$ on~$\BBR^3$ is Poisson\/-\/cohomology trivial.

\noindent$\bullet$\quad The equivalence class $\smash{\vec{X}}\mod\schouten{P,H}$ of trivializing vector fields $\smash{\vec{X}}$ satisfying the co\-boun\-d\-ary condition $\Or(\gamma_3)(P^{\otimes^4})=\schouten{P,\smash{\vec{X}} }$ is represented by the following vector field with differential\/-\/polynomial coefficients $X^i([\varrho],[a])$\textup{:}
\[
    \smash{\vec{X}}=\sum\nolimits_{\vec{\imath},\vec{\jmath},\vec{k}} \veps^{\vec{\imath}}\veps^{\vec{\jmath}}\veps^{\vec{k}}\cdot X_{\vec{\imath}\,\vec{\jmath}\,\vec{k}},
\]
where 

{\footnotesize
\begin{align*}
X_{\vec{\imath}\,\vec{\jmath}\,\vec{k}} =
&+12\varrho \varrho_{x^{k_2}} \varrho_{x^{i_1}x^{j_1}} a_{x^{k_3}} a_{x^{i_2}x^{j_2}} a_{x^{i_3}x^{j_3}} \cdot \dd/\dd x^{k_1}
+ 48\varrho \varrho_{x^{j_3}} \varrho_{x^{i_1}x^{j_1}} a_{x^{k_3}} a_{x^{i_2}x^{j_2}} a_{x^{i_3}x^{k_1}} \cdot \dd/\dd x^{k_2} \\
{}&+8\varrho_{x^{j_2}} \varrho_{x^{i_1}x^{k_1}} \varrho_{x^{i_2}x^{k_2}} a_{x^{i_3}} a_{x^{j_3}} a_{x^{k_3}} \cdot \dd/\dd x^{j_1}
- 40\varrho_{x^{i_3}} \varrho_{x^{j_2}} \varrho_{x^{i_1}x^{k_1}} a_{x^{j_3}} a_{x^{k_3}} a_{x^{i_2}x^{k_2}} \cdot \dd/\dd x^{j_1} \\
&+8\varrho_{x^{i_3}} \varrho_{x^{j_2}} \varrho_{x^{k_3}} a_{x^{j_3}} a_{x^{i_1}x^{k_1}} a_{x^{i_2}x^{k_2}} \cdot \dd/\dd x^{j_1} 
{}+ 24\varrho_{x^{j_2}} \varrho_{x^{k_3}} \varrho_{x^{i_1}x^{k_1}} a_{x^{i_3}} a_{x^{j_3}} a_{x^{j_1}x^{k_2}} \cdot \dd/\dd x^{i_2} \\
&-\underline{12\varrho^2 \varrho_{x^{k_2}} a_{x^{i_1}x^{j_1}} a_{x^{i_2}x^{j_2}} a_{x^{i_3}x^{j_3}x^{k_3}} \cdot \dd/\dd x^{k_1}}
+ \underline{24\varrho \varrho_{x^{j_2}} \varrho_{x^{k_1}} a_{x^{k_2}} a_{x^{i_1}x^{j_1}} a_{x^{i_3}x^{j_3}x^{k_3}} \cdot \dd/\dd x^{i_2}} \\
{}&-36\varrho \varrho_{x^{i_2}} \varrho_{x^{j_2}} a_{x^{k_2}} a_{x^{i_1}x^{j_1}} a_{x^{i_3}x^{j_3}x^{k_3}} \cdot \dd/\dd x^{k_1}
+ 8\varrho_{x^{i_2}} \varrho_{x^{j_1}} \varrho_{x^{k_1}} a_{x^{j_2}} a_{x^{k_2}} a_{x^{i_3}x^{j_3}x^{k_3}} \cdot \dd/\dd x^{i_1} \\
&-\underline{8\varrho_{x^{j_1}} \varrho_{x^{k_1}} \varrho_{x^{i_3}x^{j_3}x^{k_3}} a_{x^{i_2}} a_{x^{j_2}} a_{x^{k_2}} \cdot \dd/\dd x^{i_1}}.
\end{align*}

}

\noindent%
There are eleven terms in the marker\/-\/polynomial for $X_{\vec{\imath}\,\vec{\jmath}\,\vec{k}}$ but only the three underlined terms survive when the vector field $\smash{\vec{X}}$ acts on the Casimir~$a$\textup{;} the rest contributes to the velocity~$\dot{a}$ with zero markers. 

\noindent
$\bullet$\quad We verify that the velocity $\dot{a} = 4\Or(\gamma_3)(P,P,P,a)$, which we obtain by inserting the Casimir~$a(x,y,z)$
into (consecutively, each vertex of) the tetrahedron in Eq.~\eqref{EqSpeedCasimir}, is equal to the speed at which the scalar function~$a(x,y,z)$ changes at a point $\bx=(x,y,z)$ of~$\mathbb{R}^3$ under a local reparametrization of local coordinates along the integral trajectories of the vector field~$\vec{Y}=-\vec{X}$: $\dot{a} = -\vec{X}(a)$.\\
$\bullet$\quad Likewise, we verify that if
\[
\dot{a} = -\vec{X}(a) = \schouten{a, \vec{X}}\quad \text{and if}\quad \dot{\varrho} \cdot \partial_x \wedge \partial_y \wedge \partial_z = -\schouten{\varrho \cdot \partial_x \wedge \partial_y \wedge \partial_z, \vec{X}},
\]
then the Nambu\/-\/determinant Poisson bivector $P = \schouten{\varrho \cdot \partial_x \wedge \partial_y \wedge \partial_z, a}$ evolves by $\dot{P} = \schouten{P, \vec{X}} = -\lshad \vec{X},P \rshad$ over~$\mathbb{R}^3$ (according to the Leibniz rule shape~\eqref{JacForSchoutenAsLeibniz} of the Jacobi identity for the Schouten bracket).
\end{theor}

\begin{proof}
All these claims are established by direct calculation, as soon as the formula of vector field $\vec{X}$ is known.

To obtain this representative~$\vec{X}$ encoded by using the Civita symbols, we introduce a new graph calculus in which the Casi\-mir(s)~$a_i$ and the product of~$\varrho$ times Civita symbol are resolved to different vertices; from each vertex with~$\varrho$ and Civita symbol there are $d$~ordered edges outgoing, and tadpoles are allowed.
Now, to inspect a trivialization of the tetrahedral flow on the space of generalized Nambu\/-\/determinant Poisson structures over~$\mathbb{R}^3$, we make an ansatz with 
such ``micro\/-\/graphs'' on one sink with in\/-\/degree~$1$, three aerial vertices for $\varrho\,\varepsilon^{i_1 i_2 i_3}$ with three outgoing edges, and three aerial vertices with a copy of the Casimir~$a$ and no outgoing edges.
(We generated all these 
graphs by using \textsf{nauty}~\cite{nauty}.
)
Every graph at hand was interpreted as a differential polynomial because, we recall, ordered outgoing edges denote contractions of the respective upper indices in the Civita symbols and derivatives with respect to the chosen coordinates on~$\mathbb{R}^3$.
To each graph which result in a not identically zero differential polynomial in~$\varrho$,\ $a$ and the content of the sink (as we deal with a $1$-\/vector) we attach an undetermined coefficient.
In the frames of differential calculus (but not any longer within the diagrammatic algebra of graphs) we write down the equation $\dot{P} = \schouten{P, \vec{X}}$ with the tetrahedral flow of Nambu\/-\/Poisson structures in the left-hand side, and we solve the linear algebraic system for the undetermined coefficients.
A solution~$\vec{X}$ is reported in Theorem~\ref{ThG3R3VF}.
\end{proof}

In a subsequent publication~\cite{skew22} we explore the new formalism of ``micro\/-\/graphs'' in more detail: in particular, we study the nature of identities at the level of graphs if the corresponding equalities are known to hold for differential polynomials.

\begin{rem}
There can be no trivializing vector field~$\vec{X}$ in Theorem~\ref{ThG3R3VF} without tadpoles, i.e.\ there is no solution~$\vec{X}$ without micro\/-\/graph vertices in which an arrow from a Civita symbol acts back on its co\/-\/multiple~$\varrho$.
\end{rem}

\begin{example}
In the above formula of $X_{\vec{i}\vec{j}\vec{k}}$, the last term with a third derivative $\varrho_{x^{i_3}x^{j_3}x^{k_3}}$ manifests a tadpole in the graph realization.
\end{example}

\subsection{Open problems about the graph flows and their trivializing vector fields $\smash{\vec{X}}([\varrho],[\ba])$} \label{SecOpenPrb}
The study of Kontsevich flows --\,for the tetrahedral and pentagon\/-\/wheel graph cocycles (or higher vertex number cocycles $\gamma_7$,\ $[\gamma_3,\gamma_5]$,\ $\gamma_9$, etc.)\,-- restricted to the spaces of generalized ($\varrho\not\equiv1$) Nambu\/-\/determinant Poisson brackets $P([\varrho],[\ba])$ on~$\BBR^3$ and~$\BBR^4$ (or higher\/-\/dimensional affine spaces~$\BBR^d$) is, first of all, a source of combinatorial and algorithmic problems about finding the explicit shape of the objects. In particular, such is the task to collapse formulae, originally derived within the graph language, by using the Civita symbols. The other set of problems concerns the geometric nature and properties of the objects; such are the construction of the trivializing vector fields and explanation of the deeper symmetry in the choice of marker\/-\/monomials under the sums with Civita symbols. Let us summarize these problems in the order how they naturally emerge.

\begin{open}[$\dot{\varrho}$ in $\gamma_3$-flow over~$\BBR^4$]\label{PrbG3R4RhoDotSkew}
Represent the known velocity $\dot{\varrho}([\varrho],[\ba])$ for the tetrahedral $\gamma_3$-flow over~$\BBR^4$ by using three Civita symbols. Does the choice of marker\/-\/monomials enjoy the extra symmetry which is revealed in Proposition~\ref{PropHyperSymG3R3} for the $\gamma_3$-flow over~$\BBR^3$ and in Proposition~\ref{PropHyperSymG3R4A1A2} for~$\dot{\ba}$ over~$\BBR^4$?
\end{open}

\begin{open}[$\dot{\varrho},\dot{a}$ in $\gamma_5$-flow over~$\BBR^3$]\label{PrbG5R3AdotRhoDotSkew}
Represent the known velocities $\dot{\varrho}([\varrho],[a])$ and $\dot{a}([\varrho],[a])$ for the pentagon\/-\/wheel $\gamma_5$-flow over~$\BBR^3$ by using five Civita symbols. Does the extra symmetry persist for the marker\/-\/monomials in either velocity?
\end{open}

\begin{open}[$\smash{\vec{X}}$ for $\gamma_3$-flow on~$\BBR^4$ with $\varrho\equiv1$]\label{PrbVFG3R4Rho1}
Inspect whether the restriction of the tetrahedral $\gamma_3$-flow to the space of Nambu\/-\/determinant Poisson structures $P([\ba])$ on~$\BBR^4$ with $\varrho\equiv1$ is trivial in the second Poisson cohomology. --- Let us presume that there exists a trivializing vector field $\smash{\vec{X}}([\ba])$ with differential\/-\/polynomial coefficients. If it actually does, represent the coefficients --\,of possibly another vector field~$\smash{\vec{Y}}$ from the coset $\smash{\vec{X}}\mod\schouten{P,\cdot}$\,-- by using three Civita symbols on~$\BBR^4$. Do the marker\/-\/monomials in~$\vec{Y}([\ba])$ enjoy the extra symmetry?
\end{open}

\begin{open}[$\smash{\vec{X}}$ for $\gamma_3$-flow on~$\BBR^4$ with $\varrho\not\equiv1$]\label{PrbVFG3R4GenericRho}
Extend and solve Problem \ref{PrbVFG3R4Rho1} in the general case $\varrho\not\equiv1$ on~$\BBR^4$, now for the trivializing vector field~$\smash{\vec{X}}([\varrho],[\ba])$.
\end{open}

\begin{open}[$\smash{\vec{X}}$ for $\gamma_5$-flow on~$\BBR^3$]\label{PrbVFG5R3GenericRho}
Solve the trivialization problem --\,fully analogous to the above Problems~\ref{PrbVFG3R4Rho1}--\ref{PrbVFG3R4GenericRho}\,-- for the pentagon\/-\/wheel $\gamma_5$-flow over~$\BBR^3$. If it exists, the trivializing vector field $\vec{Y}([\varrho],[a])$ from the coset $\smash{\vec{X}}\mod\schouten{P,\cdot}$ (defined modulo Hamiltonian vector fields) will again be realizable by using five Civita symbols on~$\BBR^3$.
\end{open}

\begin{open} \label{PrbVFviaGraphsNambu}
Can the trivializing vector fields $\smash{\vec{X}}([\varrho],[\ba])$ be constructed --\,for nontrivial graph cocycles $\gamma$\,-- and induce the graph flows $\dot{P}=\Or(\gamma)(P^{\otimes^n})$ on the spaces of Nambu\/-\/determinant Poisson brackets $P(\varrho,[\ba])$ directly from the graph cocycles~$\gamma$ on $n$~vertices \textit{and} from the properties of the particular Poisson geometry of the Nambu brackets with global Casimirs? In other words, what are the marker\/-\/monomials for $\vec{Y}\in\vec{X}\mod\schouten{P,\cdot}$ as differential\/-\/geometric objects? (Note that the knowledge of the vector field $\smash{\vec{X}}([\varrho],[\ba])$ as the parent object for the Lie derivative~$\mathrm{L}_{\smash{\vec{X}}}$ and for $\dot{P}=\schouten{P,\smash{\vec{X}}}$ is enough to calculate~$\dot{\ba}$ and~$\dot{\varrho}$.)
\end{open}

\begin{open} \label{PrbVectorFieldsGammaNormal}
Is there a relation between the (pseudo)\/group of diffeomorphisms generated by the highly nonlinear vector fields $\smash{\vec{X}}([\varrho],[\ba])$ from the graph cocycle flows and, on the other hand, the (local) diffeomorphisms $\bx\rightleftarrows\bX'$ that map (by blowing up the local coordinates) the cells bounded by $\varrho(\bx)=0$ in~$\BBR^d$ to the domains on which $\varrho'(\bx')\equiv\pm1$?
\end{open}

In conclusion, we note that whenever they are Poisson\/-\/cohomology trivial (as we observe so far in all the cases), the nontrivial graph cocycle flows on the spaces of generalized Nambu\/-\/determinant Poisson brackets $P(\varrho,[\ba])$ not only preserve the symplectic foliation (dictated by the Casimirs~$\ba$) by merely reparametrizing the coordinate description of points still not anyhow displacing the symplectic leaves, but also preserve the tiling of the affine space~$\BBR^d$ with respect to the zero locus of the inverse density~$\varrho$ in $P(\varrho,[\ba])$. Both the foliation and 
tiling are thus rigid under the graph cocycle flows.

\appendix
\section{A class of (non)polynomial Poisson brackets on~$\BBR^d$ without global polynomial Casimir} \label{AppMKPoisson}
\noindent First, let us recall a particular construction of homogeneous polynomial\/-\/coefficient Poisson brackets on~$\BBR^d$ with Cartesian coordinates $x^1$,\ $\ldots$,\ $x^d$.

Denote by $\smash{\vec{E}}$ the Euler vector field, $\smash{\vec{E}}=\sum_i x^i\cdot\partial/\partial x^i$, and consider another nonzero vector field $\smash{\vec{V}}=\sum_j V^j(x^1$,\ $\ldots$,\ $x^d)\cdot\partial/\partial x^j$ with homogeneous polynomial coefficients~$V^j$ of total degree $k\gg1$ (conveniently starting at $k=2$). This homogeneity assumption implies that $\smash{\vec{E}}(\smash{\vec{V}})=k\cdot\smash{\vec{V}}$ and $\smash{\vec{V}} (\smash{\vec{E}}) = 1\cdot\vec{V}$, whence $[\smash{\vec{E}},\smash{\vec{V}}]=(k-1)\cdot\smash{\vec{V}}$.

By definition, put $P\mathrel{{:}{=}}\smash{\vec{V}}\wedge\smash{\vec{E}}$; this is a bi\/-\/vector with homogeneous\/-\/polynomial coefficients (of degree $k+1$).

\begin{lemma} \label{LemmaMKPoisson}
All such bi\/-\/vectors $P=\smash{\vec{V}}\wedge\smash{\vec{E}}$ on~$\BBR^d$ are Poisson.
\end{lemma}

\begin{proof}
Let us calculate the Schouten bracket $\schouten{P,P}=\schouten{\smash{\vec{V}}\wedge\smash{\vec{E}}, \smash{\vec{V}}\wedge\smash{\vec{E}}}$ by using its inductive definition for decomposable multi\/-\/vectors and thus, reducing it to the calculation of commutators for $1$-vector fields:
\begin{multline*}
    \schouten{\vec{V}\wedge\vec{E},\vec{V}\wedge\vec{E}} = \vec{V}\wedge[\smash{\vec{E}},\vec{V}]\wedge\vec{E} -\vec{V}\wedge[\smash{\vec{E}},\smash{\vec{E}}]\wedge\vec{V} -\smash{\vec{E}}\wedge[\vec{V},\vec{V}]\wedge\vec{E} +\smash{\vec{E}}\wedge[\vec{V},\smash{\vec{E}}]\wedge\vec{V} \\ 
    = 2\vec{V}\wedge[\smash{\vec{E}},\vec{V}]\wedge\vec{E}= 2(k-1)\cdot\vec{V}\wedge\vec{V}\wedge\vec{E}\equiv0.
\end{multline*}
This proves that the Jacobi identity $\tfrac{1}{2}\schouten{P,P}=0$ holds, so $P$~is Poisson.
\end{proof}

\begin{rem} \label{RemMKPoissonNonPolynomial}
The above construction of Poisson bi\/-\/vectors $P=\smash{\vec{V}}\wedge\smash{\vec{E}}$ naturally extends to homogeneous vector fields~$\smash{\vec{V}}$ (possibly not on the entire~$\BBR^d$) with not necessarily polynomial coefficients but still satisfying the condition $\smash{\vec{E}}(\smash{\vec{V}})=\lambda\cdot\smash{\vec{V}}$ with $\lambda\neq 0,1$.
\end{rem}

Now, let us produce a family of such Poisson bi\/-\/vectors $P=\smash{\vec{V}}\wedge\smash{\vec{E}}$ (with homogeneous polynomial coefficients) which do not admit any global non\/-\/constant polynomial Casimirs~--- and this is in contrast with 
the Nambu class~\eqref{EqBiVectNambu} 
for polynomial parameters~$\ba=(a_1$,\ $\ldots$,\ $a_{d-2})$.

Indeed, suppose that there is a polynomial Casimir $a$ for $P=\smash{\vec{V}}\wedge\smash{\vec{E}}$ as above.%
\footnote{\label{FootCasimirHomogeneous}%
From the homogeneity of objects it is readily seen that if $a$~is polynomial, then it is homogeneous of some degree~$D$: 
$\smash{\vec{E}}(a)=D\cdot a$).}
By the definition of Casimir, we have that
\[
\schouten{P,a}=\schouten{\vec{V}\wedge\vec{E},a}=\vec{V}\cdot\vec{E}(a)-\smash{\vec{E}}\cdot\vec{V}(a)=0,
\]
whence we obtain the system of PDE: for each $i$ running from $1$ to $d$, the Casimir $a$ satisfies the equation
\[
V^i\cdot\sum\nolimits_j x^j\cdot\partial a/\partial x^j= x^i\cdot\sum\nolimits_j V^j\cdot\partial a/\partial x^j.
\]
An infinite family of counterexamples is now produced by taking the vector fields~$\smash{\vec{V}}$ with coefficients $V^i\mathrel{{:}{=}}(x^i)^k$ for $k\geqslant2$. Indeed, we obtain that
\[
(x^i)^{k-1}\cdot\sum\nolimits_\ell x^\ell\cdot\partial a/\partial x^\ell= \sum\nolimits_j (x^j)^k\cdot\partial a/\partial x^j,
\]
and the Casimir~$a$ is by assumption polynomial in all $x^{j'}$ for $j'\not=i$ in particular. With respect to every $x^{j'}$ at $j'\not=i$ for a fixed $i$, $1\leqslant i\leqslant d$, the degree of the left\/-\/hand side, viewed as a polynomial in $x^{j'}$, is strictly not equal to that degree of the right\/-\/hand side (as $k>1$) unless $\partial a/\partial x^{j'}\equiv0$ for all $j'\not=i$. Cycling over all the equations indexed by $i$ in the system, we conclude that every polynomial Casimir~$a$ for the Poisson bi\/-\/vector $P=\smash{\vec{V}}\wedge\smash{\vec{E}}$ with $V^i=(x^i)^k$ is a constant over~$\BBR^d$.%
\qed

\begin{open}\label{OpenPrbNonNambu}
Is it true that Poisson bi\/-\/vectors $P=\vec{V}\wedge\vec{E}$ (see the above contruction) with homogeneous polynomial coefficients $P^{ij} \in \BBR[x^1,\ldots,x^d]$ but without global polynomial Casimir on~$\BBR^d$ are \emph{never} Nambu\/-\/type~--- or can there be a generalized Nambu\/-\/determinant Poisson bi\/-\/vector, with $P^{ij}$ from~\eqref{EqBiVectNambu} both polynomial and homogeneous, still having all non\/-\/polynomial Casimirs\,?
\end{open}

\begin{rem}\label{RemHomogPvectorCocycle}
We have just established the absence of global polynomial Casimirs, that is of polynomial Poisson $0$-\/cocycles for homogeneous Poisson bi\/-\/vectors $P=\smash{\vec{V}}\wedge\smash{\vec{E}}$. Let us recall from~\cite{sqs19} an affitmative statement about Poisson $1$-\/cocycles in this set\/-\/up. Specifically, every degree-$(k+1)$ homogeneous (provided $k\neq 1$ in~$\smash{\vec{V}}$), hence Poisson\/-\/exact Poisson bi\/-\/vector $P=\smash{\vec{V}}\wedge\smash{\vec{E}} = -(k-1)^{-1}\cdot [\![ P,\smash{\vec{E}} ]\!]$ is naturally accompanied with a set of Poisson $1$-\/cocycles $\smash{\vec{Z}}\in\ker [\![P,\cdot]\!]$. Their universal construction, again based on the use of Kontsevich's graph cocycles $\gamma=\sum_\ell c_\ell\cdot\gamma_\ell$ on $n$~vertices and $2n-2$ edges in each term~$\gamma_\ell$, is introduced in~\cite{sqs19}; in practice, the construction goes in parallel with Proposition~\ref{PropMKadotGraphs} (see above p.~\pageref{PropMKadotGraphs}): the formula is~$\smash{\vec{Z}} = \Or(\gamma)\bigl( P^{\otimes^{n-1}}\otimes\smash{\vec{E}} \bigr)$.
\end{rem}

\subsection*{Acknowledgements}
The authors thank the anonymous referee for helpful advice and constructive criticism;
the first and last authors thank M.\,Kontsevich for helpful advice and discussions;
R.B.\ thanks B.\,Pym for highlighting that Nambu\/-\/determinant Poisson structures are derived brackets.
The research of R.B.\ was supported by project~$5020$ at the Institute of Mathematics,
Johannes Gutenberg\/--\/Uni\-ver\-si\-t\"at Mainz and by CRC-326 grant GAUS ``Geometry and Arithmetic of Uniformized Structures''.
The travel of A.
K.\ was partially supported by project 135110 at the Bernoulli Institute, University of Groningen.
R.\,Buring and A.\,
Kiselev are grateful to the~IH\'{E}S for hospitality and financial support.


\newpage
\section{$\dot{a}$ and $\dot{\varrho}$ for the $\gamma_3$-flow over~$\BBR^3$} \label{AppARhoDotG3}\enlargethispage{\baselineskip}
{\tiny
\begin{verbatim}
adot = -12*rho^2*a_x*rho_y*a_xy*a_zz*a_xyz+12*rho^2*a_x*rho_y*a_xy*a_xz*a_yzz+12*rho^2*a_x*rho_y*a_xy
*a_xzz*a_yz-12*rho^2*a_x*rho_y*a_xz*a_yy*a_xzz+12*rho^2*a_x*rho_y*a_xz*a_yz*a_xyz-12*rho^2*a_x
*rho_y*a_yzz*a_yz*a_xx+6*rho^2*a_x*rho_y*a_xx*a_zzz*a_yy+6*rho^2*a_x*rho_y*a_xx*a_zz*a_yyz
+6*rho^2*a_x*rho_y*a_zz*a_yy*a_xxz-12*rho^2*a_x*rho_z*a_xy*a_yz*a_xyz-12*rho^2*a_x*rho_z*a_xy
*a_xz*a_yyz+12*rho^2*a_x*rho_z*a_xy*a_zz*a_xyy-12*rho^2*a_x*rho_z*a_xz*a_xyy*a_yz+12*rho^2
*a_x*rho_z*a_xz*a_yy*a_xyz+12*rho^2*a_x*rho_z*a_yz*a_xx*a_yyz-6*rho^2*a_x*rho_z*a_xx*a_zz
*a_yyy-6*rho^2*a_x*rho_z*a_xx*a_yzz*a_yy-6*rho^2*a_x*rho_z*a_zz*a_yy*a_xxy-12*rho^2*a_y*rho_x
*a_xy*a_xz*a_yzz-12*rho^2*a_y*rho_x*a_xy*a_xzz*a_yz+12*rho^2*a_y*rho_x*a_xy*a_zz*a_xyz
-12*rho^2*a_y*rho_x*a_xz*a_yz*a_xyz+12*rho^2*a_y*rho_x*a_xz*a_yy*a_xzz+12*rho^2*a_y*rho_x
*a_yzz*a_yz*a_xx-6*rho^2*a_y*rho_x*a_xx*a_zz*a_yyz-6*rho^2*a_y*rho_x*a_xx*a_zzz*a_yy-6*rho^2
*a_y*rho_x*a_zz*a_yy*a_xxz+12*rho^2*a_y*rho_z*a_xy*a_xz*a_xyz+12*rho^2*a_y*rho_z*a_xy*a_yz
*a_xxz-12*rho^2*a_y*rho_z*a_xy*a_zz*a_xxy+12*rho*a_x*a_y*rho_z*rho_x*a_yy*a_xzz-24*rho*a_x
*a_y*rho_z*rho_x*a_yz*a_xyz+12*rho*a_x*a_y*rho_z*rho_x*a_zz*a_xyy+24*rho*a_x*a_y*rho_z*rho_y
*a_xz*a_xyz-12*rho*a_x*a_y*rho_z*rho_y*a_xx*a_yzz-12*rho*a_x*a_y*rho_z*rho_y*a_zz*a_xxy
+24*rho*a_x*a_z*rho_y*rho_x*a_yz*a_xyz-12*rho*a_x*a_z*rho_y*rho_x*a_zz*a_xyy-12*rho*a_x*a_z
*rho_y*rho_x*a_yy*a_xzz-24*rho*a_x*a_z*rho_z*rho_y*a_xyz*a_xy+12*rho*a_x*a_z*rho_z*rho_y*a_xx
*a_yyz+12*rho*a_x*a_z*rho_z*rho_y*a_xxz*a_yy+12*rho*a_z*a_y*rho_x*rho_y*a_xx*a_yzz-24*rho*a_z
*a_y*rho_x*rho_y*a_xz*a_xyz+12*rho*a_z*a_y*rho_x*rho_y*a_zz*a_xxy+24*rho*a_z*a_y*rho_x*rho_z
*a_xyz*a_xy-12*rho*a_z*a_y*rho_x*rho_z*a_xx*a_yyz-12*rho*a_z*a_y*rho_x*rho_z*a_xxz*a_yy
-6*rho^2*a_x*rho_y*a_zzz*a_xy^2-6*rho^2*a_x*rho_y*a_xz^2*a_yyz-6*rho^2*a_x*rho_y*a_yz^2*a_xxz
+6*rho^2*a_x*rho_z*a_yzz*a_xy^2+6*rho^2*a_x*rho_z*a_xz^2*a_yyy+6*rho^2*a_x*rho_z*a_yz^2*a_xxy
+6*rho^2*a_y*rho_x*a_zzz*a_xy^2+6*rho^2*a_y*rho_x*a_xz^2*a_yyz+6*rho^2*a_y*rho_x*a_yz^2*a_xxz
-6*rho^2*a_y*rho_z*a_xzz*a_xy^2-6*rho^2*a_y*rho_z*a_xz^2*a_xyy-6*rho^2*a_y*rho_z*a_yz^2*a_xxx
-6*rho^2*a_z*rho_x*a_yzz*a_xy^2-6*rho^2*a_z*rho_x*a_xz^2*a_yyy-6*rho^2*a_z*rho_x*a_yz^2*a_xxy
+6*rho^2*a_z*rho_y*a_xzz*a_xy^2+6*rho^2*a_z*rho_y*a_xz^2*a_xyy+6*rho^2*a_z*rho_y*a_yz^2*a_xxx
-6*rho*a_x^2*rho_y^2*a_zz*a_xyz-6*rho*a_x^2*rho_y^2*a_xy*a_zzz+6*rho*a_x^2*rho_y^2*a_xz*a_yzz
+6*rho*a_x^2*rho_y^2*a_xzz*a_yz-6*rho*a_x^2*rho_z^2*a_xyy*a_yz-6*rho*a_x^2*rho_z^2*a_xy*a_yyz
+6*rho*a_x^2*rho_z^2*a_xz*a_yyy+6*rho*a_x^2*rho_z^2*a_yy*a_xyz-6*rho*a_y^2*rho_x^2*a_xzz*a_yz
+6*rho*a_y^2*rho_x^2*a_zz*a_xyz+6*rho*a_y^2*rho_x^2*a_xy*a_zzz-6*rho*a_y^2*rho_x^2*a_xz*a_yzz
-6*rho*a_y^2*rho_z^2*a_xxx*a_yz+6*rho*a_y^2*rho_z^2*a_xxz*a_xy-6*rho*a_y^2*rho_z^2*a_xx*a_xyz
+6*rho*a_y^2*rho_z^2*a_xxy*a_xz+6*rho*a_z^2*rho_x^2*a_xy*a_yyz-6*rho*a_z^2*rho_x^2*a_xz*a_yyy
-6*rho*a_z^2*rho_x^2*a_yy*a_xyz+6*rho*a_z^2*rho_x^2*a_xyy*a_yz+6*rho*a_z^2*rho_y^2*a_xxx*a_yz
+6*rho*a_z^2*rho_y^2*a_xx*a_xyz-6*rho*a_z^2*rho_y^2*a_xxy*a_xz-6*rho*a_z^2*rho_y^2*a_xxz*a_xy
+6*a_x^2*a_y*rho_x*a_zzz*rho_y^2+6*a_x^2*a_y*rho_x*a_yyz*rho_z^2+12*a_x^2*a_y*a_xyz*rho_y
*rho_z^2-6*a_x^2*a_y*a_xzz*rho_z*rho_y^2-6*a_x^2*a_z*rho_x*rho_y^2*a_yzz-6*a_x^2*a_z*rho_x
*rho_z^2*a_yyy+6*a_x^2*a_z*a_xyy*rho_z^2*rho_y-12*a_x^2*a_z*a_xyz*rho_z*rho_y^2+6*a_x*a_y^2
*rho_x^2*rho_z*a_yzz-6*a_x*a_y^2*rho_x^2*a_zzz*rho_y-12*a_x*a_y^2*rho_x*a_xyz*rho_z^2-6*a_x
*a_y^2*a_xxz*rho_z^2*rho_y+6*a_x*a_z^2*rho_x^2*rho_z*a_yyy-6*a_x*a_z^2*rho_x^2*a_yyz*rho_y
+12*a_x*a_z^2*rho_x*a_xyz*rho_y^2+6*a_x*a_z^2*a_xxy*rho_z*rho_y^2+12*a_z*a_y^2*rho_x^2*rho_z
*a_xyz+6*a_z*a_y^2*rho_x^2*rho_y*a_xzz-6*a_z*a_y^2*rho_x*a_xxy*rho_z^2+6*a_z*a_y^2*a_xxx*rho_y
*rho_z^2-6*a_z^2*a_y*rho_x^2*a_xyy*rho_z-12*a_z^2*a_y*rho_x^2*a_xyz*rho_y+6*a_z^2*a_y*rho_x
*rho_y^2*a_xxz-6*a_z^2*a_y*a_xxx*rho_z*rho_y^2+6*a_x^3*a_yzz*rho_y^2*rho_z-6*a_x^3*a_yyz
*rho_z^2*rho_y-6*a_x^2*a_y*a_xyy*rho_z^3+6*a_x^2*a_z*a_xzz*rho_y^3+6*a_x*a_y^2*a_xxy*rho_z^3
-6*a_x*a_z^2*a_xxz*rho_y^3+6*a_y^3*a_xxz*rho_x*rho_z^2-6*a_y^3*a_xzz*rho_x^2*rho_z-6*a_z*a_y^2
*a_yzz*rho_x^3+6*a_z^2*a_y*a_yyz*rho_x^3+6*a_z^3*a_xyy*rho_x^2*rho_y-6*a_z^3*a_xxy*rho_x
*rho_y^2-12*rho^2*a_y*rho_z*a_xz*a_xxz*a_yy+12*rho^2*a_y*rho_z*a_xz*a_yz*a_xxy-12*rho^2*a_y
*rho_z*a_yz*a_xx*a_xyz+6*rho^2*a_y*rho_z*a_xx*a_zz*a_xyy+6*rho^2*a_y*rho_z*a_xx*a_yy*a_xzz
+6*rho^2*a_y*rho_z*a_zz*a_yy*a_xxx+12*rho^2*a_z*rho_x*a_xy*a_xz*a_yyz+12*rho^2*a_z*rho_x*a_xy
*a_yz*a_xyz-12*rho^2*a_z*rho_x*a_xy*a_zz*a_xyy-12*rho^2*a_z*rho_x*a_xz*a_yy*a_xyz+12*rho^2*a_z
*rho_x*a_xz*a_xyy*a_yz-12*rho^2*a_z*rho_x*a_yz*a_xx*a_yyz+6*rho^2*a_z*rho_x*a_xx*a_zz*a_yyy
+6*rho^2*a_z*rho_x*a_xx*a_yzz*a_yy+6*rho^2*a_z*rho_x*a_zz*a_yy*a_xxy-12*rho^2*a_z*rho_y*a_xy
*a_yz*a_xxz+12*rho^2*a_z*rho_y*a_xy*a_zz*a_xxy-12*rho^2*a_z*rho_y*a_xy*a_xz*a_xyz-12*rho^2*a_z
*rho_y*a_xz*a_yz*a_xxy+12*rho^2*a_z*rho_y*a_xz*a_xxz*a_yy+12*rho^2*a_z*rho_y*a_yz*a_xx*a_xyz
-6*rho^2*a_z*rho_y*a_xx*a_zz*a_xyy-6*rho^2*a_z*rho_y*a_xx*a_yy*a_xzz-6*rho^2*a_z*rho_y*a_zz
*a_yy*a_xxx+6*rho*a_x^2*rho_x*rho_y*a_zz*a_yyz+6*rho*a_x^2*rho_x*rho_y*a_zzz*a_yy-12*rho*a_x^2
*rho_x*rho_y*a_yzz*a_yz-6*rho*a_x^2*rho_x*rho_z*a_yzz*a_yy-6*rho*a_x^2*rho_x*rho_z*a_zz*a_yyy
+12*rho*a_x^2*rho_x*rho_z*a_yyz*a_yz+6*rho*a_x^2*rho_z*rho_y*a_zz*a_xyy+12*rho*a_x^2*rho_z
*rho_y*a_yzz*a_xy-6*rho*a_x^2*rho_z*rho_y*a_yy*a_xzz-12*rho*a_x^2*rho_z*rho_y*a_xz*a_yyz
-6*rho*a_x*a_y*rho_x^2*a_zz*a_yyz-6*rho*a_x*a_y*rho_x^2*a_zzz*a_yy+12*rho*a_x*a_y*rho_x^2
*a_yzz*a_yz+6*rho*a_x*a_y*rho_y^2*a_zz*a_xxz-12*rho*a_x*a_y*rho_y^2*a_xzz*a_xz+6*rho*a_x*a_y
*rho_y^2*a_xx*a_zzz-6*rho*a_x*a_y*rho_z^2*a_xxz*a_yy+6*rho*a_x*a_y*rho_z^2*a_xx*a_yyz+12*rho
*a_x*a_y*rho_z^2*a_yz*a_xxy-12*rho*a_x*a_y*rho_z^2*a_xz*a_xyy-12*rho*a_x*a_z*rho_x^2*a_yyz
*a_yz+6*rho*a_x*a_z*rho_x^2*a_zz*a_yyy+6*rho*a_x*a_z*rho_x^2*a_yzz*a_yy-6*rho*a_x*a_z*rho_y^2
*a_xx*a_yzz+6*rho*a_x*a_z*rho_y^2*a_zz*a_xxy+12*rho*a_x*a_z*rho_y^2*a_xzz*a_xy-12*rho*a_x*a_z
*rho_y^2*a_yz*a_xxz+12*rho*a_x*a_z*rho_z^2*a_xyy*a_xy-6*rho*a_x*a_z*rho_z^2*a_xxy*a_yy-6*rho
*a_x*a_z*rho_z^2*a_xx*a_yyy-6*rho*a_y^2*rho_x*rho_y*a_zz*a_xxz+12*rho*a_y^2*rho_x*rho_y*a_xzz
*a_xz-6*rho*a_y^2*rho_x*rho_y*a_xx*a_zzz+12*rho*a_y^2*rho_x*rho_z*a_yz*a_xxz+6*rho*a_y^2*rho_x
*rho_z*a_xx*a_yzz-6*rho*a_y^2*rho_x*rho_z*a_zz*a_xxy-12*rho*a_y^2*rho_x*rho_z*a_xzz*a_xy-12
*rho*a_y^2*rho_z*rho_y*a_xz*a_xxz+6*rho*a_y^2*rho_z*rho_y*a_xxx*a_zz+6*rho*a_y^2*rho_z*rho_y
*a_xzz*a_xx+12*rho*a_z*a_y*rho_x^2*a_xz*a_yyz-6*rho*a_z*a_y*rho_x^2*a_zz*a_xyy-12*rho*a_z*a_y
*rho_x^2*a_yzz*a_xy+6*rho*a_z*a_y*rho_x^2*a_yy*a_xzz+12*rho*a_z*a_y*rho_y^2*a_xz*a_xxz-6*rho
*a_z*a_y*rho_y^2*a_xxx*a_zz-6*rho*a_z*a_y*rho_y^2*a_xzz*a_xx+6*rho*a_z*a_y*rho_z^2*a_yy*a_xxx
+6*rho*a_z*a_y*rho_z^2*a_xyy*a_xx-12*rho*a_z*a_y*rho_z^2*a_xy*a_xxy-12*rho*a_z^2*rho_x*rho_y
*a_yz*a_xxy+6*rho*a_z^2*rho_x*rho_y*a_xxz*a_yy-6*rho*a_z^2*rho_x*rho_y*a_xx*a_yyz+12*rho*a_z^2
*rho_x*rho_y*a_xz*a_xyy+6*rho*a_z^2*rho_x*rho_z*a_xxy*a_yy+6*rho*a_z^2*rho_x*rho_z*a_xx*a_yyy
-12*rho*a_z^2*rho_x*rho_z*a_xyy*a_xy+12*rho*a_z^2*rho_z*rho_y*a_xy*a_xxy-6*rho*a_z^2*rho_z
*rho_y*a_xyy*a_xx-6*rho*a_z^2*rho_z*rho_y*a_yy*a_xxx-12*a_x^2*a_y*rho_x*rho_y*rho_z*a_yzz
+12*a_x^2*a_z*rho_x*rho_y*rho_z*a_yyz+12*a_x*a_y^2*rho_x*rho_y*rho_z*a_xzz+12*a_x*a_z*a_y
*rho_x^2*a_yzz*rho_y-12*a_x*a_z*a_y*rho_x^2*rho_z*a_yyz-12*a_x*a_z*a_y*rho_x*rho_y^2*a_xzz
+12*a_x*a_z*a_y*rho_x*rho_z^2*a_xyy-12*a_x*a_z*a_y*a_xxy*rho_z^2*rho_y+12*a_x*a_z*a_y*a_xxz
*rho_z*rho_y^2-12*a_x*a_z^2*rho_x*rho_y*rho_z*a_xyy-12*a_z*a_y^2*rho_x*rho_y*rho_z*a_xxz+12
*a_z^2*a_y*rho_x*rho_y*rho_z*a_xxy-2*a_x^3*a_zzz*rho_y^3+2*a_x^3*a_yyy*rho_z^3+2*a_y^3*a_zzz
*rho_x^3-2*a_y^3*a_xxx*rho_z^3+2*a_z^3*a_xxx*rho_y^3-2*a_z^3*a_yyy*rho_x^3

rhodot =
-12*rho*rho_x*rho_y*a_x*a_z*rho_xyy*a_zz-12*rho*rho_x*rho_y*a_x*a_z*rho_xzz*a_yy+24*rho*rho_x
*rho_y*a_x*a_z*rho_xyz*a_yz+12*rho*rho_x*rho_y*a_x*a_xy*rho_yz*a_zz-12*rho*rho_x*rho_y*a_x
*a_xy*a_yz*rho_zz-12*rho*rho_x*rho_y*a_x*a_xz*a_yz*rho_yz+6*rho^2*rho_x*a_y*rho_zzz*a_xy^2
+6*rho^2*rho_x*a_y*rho_yyz*a_xz^2+6*rho^2*rho_x*a_y*rho_xxz*a_yz^2-6*rho^2*rho_x*a_z*rho_yzz
*a_xy^2-6*rho^2*rho_x*a_z*rho_yyy*a_xz^2-6*rho^2*rho_x*a_z*rho_xxy*a_yz^2-6*rho^2*rho_y*a_x
*rho_zzz*a_xy^2-6*rho^2*rho_y*a_x*rho_yyz*a_xz^2-6*rho^2*rho_y*a_x*rho_xxz*a_yz^2+6*rho^2
*rho_y*a_z*rho_xzz*a_xy^2+6*rho^2*rho_y*a_z*rho_xyy*a_xz^2+6*rho^2*rho_y*a_z*rho_xxx*a_yz^2
+6*rho^2*rho_z*a_x*rho_yzz*a_xy^2+6*rho^2*rho_z*a_x*rho_yyy*a_xz^2+6*rho^2*rho_z*a_x*rho_xxy
*a_yz^2-6*rho^2*rho_z*a_y*rho_xzz*a_xy^2-6*rho^2*rho_z*a_y*rho_xyy*a_xz^2-6*rho^2*rho_z*a_y
*rho_xxx*a_yz^2+6*rho*rho_x^2*a_y^2*rho_zzz*a_xy-6*rho*rho_x^2*a_y^2*rho_xzz*a_yz+6*rho
*rho_x^2*a_y^2*rho_xyz*a_zz-6*rho*rho_x^2*a_y^2*rho_yzz*a_xz-12*rho*rho_x^2*a_y*rho_xz*a_yz^2
-6*rho*rho_x^2*a_z^2*a_xz*rho_yyy-6*rho*rho_x^2*a_z^2*rho_xyz*a_yy+6*rho*rho_x^2*a_z^2*rho_xyy
*a_yz+6*rho*rho_x^2*a_z^2*a_xy*rho_yyz+12*rho*rho_x^2*a_z*rho_xy*a_yz^2+6*rho*rho_y^2*a_x^2
*rho_xzz*a_yz+6*rho*rho_y^2*a_x^2*rho_yzz*a_xz-6*rho*rho_y^2*a_x^2*rho_zzz*a_xy-6*rho*rho_y^2
*a_x^2*rho_xyz*a_zz+12*rho*rho_y^2*a_x*rho_yz*a_xz^2-6*rho*rho_y^2*a_z^2*rho_xxy*a_xz-6*rho
*rho_y^2*a_z^2*a_xy*rho_xxz+6*rho*rho_y^2*a_z^2*rho_xyz*a_xx+6*rho*rho_y^2*a_z^2*a_yz*rho_xxx
-12*rho*rho_y^2*a_z*rho_xy*a_xz^2+6*rho*rho_z^2*a_x^2*a_xz*rho_yyy+6*rho*rho_z^2*a_x^2*rho_xyz
*a_yy-6*rho*rho_z^2*a_x^2*rho_xyy*a_yz-6*rho*rho_z^2*a_x^2*a_xy*rho_yyz-12*rho*rho_z^2*a_x
*a_xy^2*rho_yz-6*rho*rho_z^2*a_y^2*a_yz*rho_xxx+6*rho*rho_z^2*a_y^2*a_xy*rho_xxz-6*rho*rho_z^2
*a_y^2*rho_xyz*a_xx+6*rho*rho_z^2*a_y^2*rho_xxy*a_xz+12*rho*rho_z^2*a_y*rho_xz*a_xy^2
-6*rho_x^3*a_z*a_y*rho_zz*a_yy+6*rho_x^3*a_z*a_y*rho_yy*a_zz-6*rho_x^2*rho_y*a_x*rho_zzz*a_y^2
-6*rho_x^2*rho_y*a_x*rho_yyz*a_z^2-6*rho_x^2*rho_y*a_y^2*a_xz*rho_zz+6*rho_x^2*rho_y*a_y^2*a_z
*rho_xzz+6*rho_x^2*rho_y*a_y^2*rho_xz*a_zz-12*rho_x^2*rho_y*a_z^2*a_y*rho_xyz-12*rho_x^2*rho_y
*a_z^2*rho_yz*a_xy+12*rho_x^2*rho_y*a_z^2*a_yz*rho_xy-6*rho_x^2*rho_y*a_z^2*rho_xz*a_yy
+6*rho_x^2*rho_y*a_z^2*a_xz*rho_yy+6*rho_x^2*rho_z*a_x*rho_yzz*a_y^2+6*rho_x^2*rho_z*a_x*a_z^2
*rho_yyy+6*rho_x^2*rho_z*a_y^2*rho_xy*a_zz+12*rho_x^2*rho_z*a_y^2*rho_yz*a_xz+12*rho_x^2*rho_z
*a_y^2*a_z*rho_xyz-6*rho_x^2*rho_z*a_y^2*a_xy*rho_zz-12*rho_x^2*rho_z*a_y^2*rho_xz*a_yz
-6*rho_x^2*rho_z*a_z^2*a_y*rho_xyy-6*rho_x^2*rho_z*a_z^2*rho_xy*a_yy+6*rho_x^2*rho_z*a_z^2
*rho_yy*a_xy-6*rho_x*rho_y^2*a_x^2*rho_yzz*a_z+6*rho_x*rho_y^2*a_x^2*rho_zzz*a_y+6*rho_x
*rho_y^2*a_x^2*a_yz*rho_zz-6*rho_x*rho_y^2*a_x^2*rho_yz*a_zz+12*rho_x*rho_y^2*a_x*a_z^2
*rho_xyz+6*rho_x*rho_y^2*a_y*a_z^2*rho_xxz-12*rho_x*rho_y^2*a_z^2*rho_xy*a_xz+6*rho_x*rho_y^2
*a_z^2*rho_yz*a_xx-6*rho_x*rho_y^2*a_z^2*a_yz*rho_xx-6*rho_x^3*a_y^2*rho_yz*a_zz-6*rho_x^3
*a_y^2*rho_yzz*a_z+6*rho_x^3*a_y^2*a_yz*rho_zz+6*rho_x^3*a_z^2*a_y*rho_yyz-6*rho_x^3*a_z^2
*a_yz*rho_yy+6*rho_x^3*a_z^2*rho_yz*a_yy+6*rho_x^2*rho_y*a_z^3*rho_xyy-6*rho_x^2*rho_z*rho_xzz
*a_y^3-6*rho_x*rho_y^2*a_z^3*rho_xxy+6*rho_x*rho_z^2*rho_xxz*a_y^3-6*rho_y^3*a_x^2*a_xz*rho_zz
+6*rho_y^3*a_x^2*a_z*rho_xzz+6*rho_y^3*a_x^2*rho_xz*a_zz-6*rho_y^3*a_z^2*a_x*rho_xxz-6*rho_y^3
*a_z^2*a_xx*rho_xz+6*rho_y^3*a_z^2*a_xz*rho_xx+6*rho_z*rho_y^2*rho_yzz*a_x^3-6*rho_z^2*rho_y
*rho_yyz*a_x^3+6*rho_z^3*a_x^2*rho_yy*a_xy-6*rho_z^3*a_x^2*a_y*rho_xyy-6*rho_z^3*a_x^2*rho_xy
*a_yy+6*rho_z^3*a_y^2*a_x*rho_xxy-6*rho_z^3*a_y^2*a_xy*rho_xx-12*rho^2*rho_x*a_y*a_xy*rho_yzz
*a_xz-12*rho^2*rho_x*a_y*a_xy*rho_xzz*a_yz+12*rho^2*rho_x*a_y*a_xy*rho_xyz*a_zz+12*rho^2*rho_x
*a_y*a_xz*rho_xzz*a_yy-12*rho^2*rho_x*a_y*a_xz*rho_xyz*a_yz+12*rho^2*rho_x*a_y*rho_yzz*a_yz
*a_xx-6*rho^2*rho_x*a_y*a_xx*rho_yyz*a_zz-6*rho^2*rho_x*a_y*a_xx*rho_zzz*a_yy-6*rho^2*rho_x
*a_y*rho_xxz*a_zz*a_yy+12*rho^2*rho_x*a_z*a_xy*rho_yyz*a_xz-12*rho^2*rho_x*a_z*a_xy*rho_xyy
*a_zz+12*rho^2*rho_x*a_z*a_xy*rho_xyz*a_yz-12*rho^2*rho_x*a_z*a_xz*rho_xyz*a_yy+12*rho^2*rho_x
*a_z*a_xz*rho_xyy*a_yz-12*rho^2*rho_x*a_z*rho_yyz*a_yz*a_xx+6*rho^2*rho_x*a_z*a_xx*rho_yyy
*a_zz+6*rho^2*rho_x*a_z*a_xx*rho_yzz*a_yy+6*rho^2*rho_x*a_z*rho_xxy*a_zz*a_yy+12*rho^2*rho_y
*a_x*a_xy*rho_yzz*a_xz+12*rho^2*rho_y*a_x*a_xy*rho_xzz*a_yz-12*rho^2*rho_y*a_x*a_xy*rho_xyz
*a_zz-12*rho^2*rho_y*a_x*a_xz*rho_xzz*a_yy+12*rho^2*rho_y*a_x*a_xz*rho_xyz*a_yz-12*rho^2*rho_y
*a_x*rho_yzz*a_yz*a_xx+6*rho^2*rho_y*a_x*a_xx*rho_yyz*a_zz+6*rho^2*rho_y*a_x*a_xx*rho_zzz*a_yy
+6*rho^2*rho_y*a_x*rho_xxz*a_zz*a_yy+12*rho^2*rho_y*a_z*a_xy*rho_xxy*a_zz-12*rho^2*rho_y*a_z
*a_xy*rho_xxz*a_yz-12*rho^2*rho_y*a_z*a_xy*a_xz*rho_xyz+12*rho^2*rho_y*a_z*a_xz*rho_xxz*a_yy
-12*rho^2*rho_y*a_z*a_xz*rho_xxy*a_yz+12*rho^2*rho_y*a_z*rho_xyz*a_yz*a_xx-6*rho^2*rho_y*a_z
*a_xx*rho_xzz*a_yy-6*rho^2*rho_y*a_z*a_xx*rho_xyy*a_zz-6*rho^2*rho_y*a_z*rho_xxx*a_zz*a_yy
-12*rho^2*rho_z*a_x*a_xy*rho_yyz*a_xz-12*rho^2*rho_z*a_x*a_xy*rho_xyz*a_yz+12*rho^2*rho_z*a_x
*a_xy*rho_xyy*a_zz+12*rho^2*rho_z*a_x*a_xz*rho_xyz*a_yy-12*rho^2*rho_z*a_x*a_xz*rho_xyy*a_yz
+12*rho^2*rho_z*a_x*rho_yyz*a_yz*a_xx-6*rho^2*rho_z*a_x*a_xx*rho_yzz*a_yy-6*rho^2*rho_z*a_x
*a_xx*rho_yyy*a_zz-6*rho^2*rho_z*a_x*rho_xxy*a_zz*a_yy-12*rho^2*rho_z*a_y*a_xy*rho_xxy*a_zz
+12*rho^2*rho_z*a_y*a_xy*rho_xxz*a_yz+12*rho^2*rho_z*a_y*a_xy*a_xz*rho_xyz-12*rho^2*rho_z*a_y
*a_xz*rho_xxz*a_yy+12*rho^2*rho_z*a_y*a_xz*rho_xxy*a_yz-12*rho^2*rho_z*a_y*rho_xyz*a_yz*a_xx
+6*rho^2*rho_z*a_y*a_xx*rho_xzz*a_yy+6*rho^2*rho_z*a_y*a_xx*rho_xyy*a_zz+6*rho^2*rho_z*a_y
*rho_xxx*a_zz*a_yy-6*rho*rho_x^2*a_x*a_y*rho_yyz*a_zz-6*rho*rho_x^2*a_x*a_y*rho_zzz*a_yy
+12*rho*rho_x^2*a_x*a_y*rho_yzz*a_yz-12*rho*rho_x^2*a_x*a_z*rho_yyz*a_yz+6*rho*rho_x^2*a_x*a_z
*rho_yyy*a_zz+6*rho*rho_x^2*a_x*a_z*rho_yzz*a_yy+12*rho*rho_x^2*a_y*a_z*rho_yyz*a_xz+6*rho
*rho_x^2*a_y*a_z*rho_xzz*a_yy-6*rho*rho_x^2*a_y*a_z*rho_xyy*a_zz-12*rho*rho_x^2*a_y*a_z
*rho_yzz*a_xy-12*rho*rho_x^2*a_y*a_xy*rho_yz*a_zz+12*rho*rho_x^2*a_y*a_xy*a_yz*rho_zz
+12*rho*rho_x^2*a_y*a_xz*a_yz*rho_yz-12*rho*rho_x^2*a_y*a_xz*rho_zz*a_yy+12*rho*rho_x^2*a_y
*rho_xz*a_yy*a_zz-12*rho*rho_x^2*a_z*a_xy*a_yz*rho_yz+12*rho*rho_x^2*a_z*a_xy*rho_yy*a_zz
+12*rho*rho_x^2*a_z*a_xz*rho_yz*a_yy-12*rho*rho_x^2*a_z*a_xz*a_yz*rho_yy-12*rho*rho_x^2*a_z
*rho_xy*a_yy*a_zz+6*rho*rho_x*rho_y*a_x^2*rho_yyz*a_zz+6*rho*rho_x*rho_y*a_x^2*rho_zzz*a_yy
-12*rho*rho_x*rho_y*a_x^2*rho_yzz*a_yz+12*rho*rho_x*rho_y*a_x*rho_xz*a_yz^2-6*rho*rho_x*rho_y
*a_y^2*rho_xxz*a_zz+12*rho*rho_x*rho_y*a_y^2*a_xz*rho_xzz-6*rho*rho_x*rho_y*a_y^2*a_xx*rho_zzz
-12*rho*rho_x*rho_y*a_y*rho_yz*a_xz^2-6*rho*rho_x*rho_y*a_z^2*rho_yyz*a_xx+6*rho*rho_x*rho_y
*a_z^2*rho_xxz*a_yy-12*rho*rho_x*rho_y*a_z^2*rho_xxy*a_yz+12*rho*rho_x*rho_y*a_z^2*rho_xyy
*a_xz+12*rho*rho_x*rho_y*a_z*rho_yy*a_xz^2-12*rho*rho_x*rho_y*a_z*rho_xx*a_yz^2-6*rho*rho_x
*rho_z*a_x^2*rho_yyy*a_zz-6*rho*rho_x*rho_z*a_x^2*rho_yzz*a_yy+12*rho*rho_x*rho_z*a_x^2
*rho_yyz*a_yz-12*rho*rho_x*rho_z*a_x*rho_xy*a_yz^2-12*rho*rho_x*rho_z*a_y^2*rho_xzz*a_xy
+6*rho*rho_x*rho_z*a_y^2*rho_yzz*a_xx-6*rho*rho_x*rho_z*a_y^2*rho_xxy*a_zz+12*rho*rho_x*rho_z
*a_y^2*rho_xxz*a_yz-12*rho*rho_x*rho_z*a_y*rho_zz*a_xy^2+12*rho*rho_x*rho_z*a_y*rho_xx*a_yz^2
-12*rho*rho_x*rho_z*a_z^2*rho_xyy*a_xy+6*rho*rho_x*rho_z*a_z^2*a_xx*rho_yyy+6*rho*rho_x*rho_z
*a_z^2*rho_xxy*a_yy+12*rho*rho_x*rho_z*a_z*a_xy^2*rho_yz-12*rho*rho_y^2*a_x*a_y*a_xz*rho_xzz
+6*rho*rho_y^2*a_x*a_y*a_xx*rho_zzz+6*rho*rho_y^2*a_x*a_y*rho_xxz*a_zz+12*rho*rho_y^2*a_x*a_z
*rho_xzz*a_xy+6*rho*rho_y^2*a_x*a_z*rho_xxy*a_zz-6*rho*rho_y^2*a_x*a_z*rho_yzz*a_xx-12*rho
*rho_y^2*a_x*a_z*rho_xxz*a_yz-12*rho*rho_y^2*a_x*a_xy*a_xz*rho_zz+12*rho*rho_y^2*a_x*a_xy
*rho_xz*a_zz-12*rho*rho_y^2*a_x*a_xz*rho_xz*a_yz+12*rho*rho_y^2*a_x*a_xx*a_yz*rho_zz-12*rho
*rho_y^2*a_x*a_xx*rho_yz*a_zz-6*rho*rho_y^2*a_z*a_y*rho_xzz*a_xx-6*rho*rho_y^2*a_z*a_y*rho_xxx
*a_zz+12*rho*rho_y^2*a_z*a_y*rho_xxz*a_xz+12*rho*rho_y^2*a_z*a_xy*rho_xz*a_xz-12*rho*rho_y^2
*a_z*a_xy*a_zz*rho_xx+12*rho*rho_y^2*a_z*rho_xx*a_yz*a_xz+12*rho*rho_y^2*a_z*a_xx*rho_xy*a_zz
-12*rho*rho_y^2*a_z*a_xx*rho_xz*a_yz-12*rho*rho_z*rho_y*a_x^2*rho_yyz*a_xz+12*rho*rho_z*rho_y
*a_x^2*rho_yzz*a_xy+6*rho*rho_z*rho_y*a_x^2*rho_xyy*a_zz-6*rho*rho_z*rho_y*a_x^2*rho_xzz*a_yy
+12*rho*rho_z*rho_y*a_x*rho_zz*a_xy^2-12*rho*rho_z*rho_y*a_x*rho_yy*a_xz^2+6*rho*rho_z*rho_y
*a_y^2*rho_xzz*a_xx+6*rho*rho_z*rho_y*a_y^2*rho_xxx*a_zz-12*rho*rho_z*rho_y*a_y^2*rho_xxz*a_xz
+12*rho*rho_z*rho_y*a_y*rho_xy*a_xz^2+12*rho*rho_z*rho_y*a_z^2*rho_xxy*a_xy-6*rho*rho_z*rho_y
*a_z^2*a_yy*rho_xxx-6*rho*rho_z*rho_y*a_z^2*rho_xyy*a_xx-12*rho*rho_z*rho_y*a_z*rho_xz*a_xy^2
+12*rho*rho_z^2*a_x*a_y*rho_xxy*a_yz-12*rho*rho_z^2*a_x*a_y*rho_xyy*a_xz-6*rho*rho_z^2*a_x*a_y
*rho_xxz*a_yy+6*rho*rho_z^2*a_x*a_y*rho_yyz*a_xx+12*rho*rho_z^2*a_x*a_z*rho_xyy*a_xy-6*rho
*rho_z^2*a_x*a_z*a_xx*rho_yyy-6*rho*rho_z^2*a_x*a_z*rho_xxy*a_yy+12*rho*rho_z^2*a_x*a_xy*a_xz
*rho_yy+12*rho*rho_z^2*a_x*a_xy*a_yz*rho_xy-12*rho*rho_z^2*a_x*rho_xy*a_xz*a_yy+12*rho*rho_z^2
*a_x*a_xx*rho_yz*a_yy-12*rho*rho_z^2*a_x*a_xx*a_yz*rho_yy-12*rho*rho_z^2*a_y*a_z*rho_xxy*a_xy
+6*rho*rho_z^2*a_y*a_z*a_yy*rho_xxx+6*rho*rho_z^2*a_y*a_z*rho_xyy*a_xx-12*rho*rho_z^2*a_y*a_xy
*rho_xy*a_xz-12*rho*rho_z^2*a_y*a_xy*a_yz*rho_xx+12*rho*rho_z^2*a_y*rho_xx*a_yy*a_xz-12*rho
*rho_z^2*a_y*a_xx*rho_xz*a_yy+12*rho*rho_z^2*a_y*a_xx*a_yz*rho_xy+12*rho_x^2*rho_y*a_x*a_y
*rho_yz*a_zz-12*rho_x^2*rho_y*a_x*a_y*a_yz*rho_zz+12*rho_x^2*rho_y*a_x*a_y*rho_yzz*a_z
+6*rho_x^2*rho_y*a_x*a_z*rho_zz*a_yy-6*rho_x^2*rho_y*a_x*a_z*rho_yy*a_zz-12*rho_x^2*rho_y*a_z
*a_y*rho_xy*a_zz+12*rho_x^2*rho_y*a_z*a_y*a_xy*rho_zz+6*rho_x^2*rho_z*a_x*a_y*rho_zz*a_yy
-6*rho_x^2*rho_z*a_x*a_y*rho_yy*a_zz-12*rho_x^2*rho_z*a_x*a_y*rho_yyz*a_z-12*rho_x^2*rho_z*a_x
*a_z*rho_yz*a_yy+12*rho_x^2*rho_z*a_x*a_z*a_yz*rho_yy+12*rho_x^2*rho_z*a_z*a_y*rho_xz*a_yy
-12*rho_x^2*rho_z*a_z*a_y*a_xz*rho_yy+12*rho_x*rho_y^2*a_x*a_y*a_xz*rho_zz-12*rho_x*rho_y^2
*a_x*a_y*a_z*rho_xzz-12*rho_x*rho_y^2*a_x*a_y*rho_xz*a_zz+12*rho_x*rho_y^2*a_x*a_z*rho_xy*a_zz
-12*rho_x*rho_y^2*a_x*a_z*a_xy*rho_zz-6*rho_x*rho_y^2*a_z*a_y*a_xx*rho_zz+6*rho_x*rho_y^2*a_z
*a_y*a_zz*rho_xx-12*rho_x*rho_z*rho_y*a_x^2*a_y*rho_yzz+12*rho_x*rho_z*rho_y*a_x^2*rho_yyz*a_z
-6*rho_x*rho_z*rho_y*a_x^2*rho_zz*a_yy+6*rho_x*rho_z*rho_y*a_x^2*rho_yy*a_zz+12*rho_x*rho_z
*rho_y*a_x*rho_xzz*a_y^2-12*rho_x*rho_z*rho_y*a_x*a_z^2*rho_xyy-6*rho_x*rho_z*rho_y*a_y^2*a_zz
*rho_xx-12*rho_x*rho_z*rho_y*a_y^2*a_z*rho_xxz+6*rho_x*rho_z*rho_y*a_y^2*a_xx*rho_zz+12*rho_x
*rho_z*rho_y*a_z^2*a_y*rho_xxy+6*rho_x*rho_z*rho_y*a_z^2*a_yy*rho_xx-6*rho_x*rho_z*rho_y*a_z^2
*rho_yy*a_xx+12*rho_x*rho_z^2*a_x*a_y*a_z*rho_xyy-12*rho_x*rho_z^2*a_x*a_y*rho_xz*a_yy
+12*rho_x*rho_z^2*a_x*a_y*a_xz*rho_yy+12*rho_x*rho_z^2*a_x*a_z*rho_xy*a_yy-12*rho_x*rho_z^2
*a_x*a_z*rho_yy*a_xy-6*rho_x*rho_z^2*a_y*a_z*a_yy*rho_xx+6*rho_x*rho_z^2*a_y*a_z*rho_yy*a_xx
+6*rho_z*rho_y^2*a_x*a_y*a_zz*rho_xx+12*rho_z*rho_y^2*a_x*a_y*a_z*rho_xxz-6*rho_z*rho_y^2*a_x
*a_y*a_xx*rho_zz+12*rho_z*rho_y^2*a_x*a_z*a_yz*rho_xx-12*rho_z*rho_y^2*a_x*a_z*rho_yz*a_xx
+12*rho_z*rho_y^2*a_z*a_y*a_xx*rho_xz-12*rho_z*rho_y^2*a_z*a_y*a_xz*rho_xx+12*rho_z^2*rho_y
*a_x*a_y*rho_yz*a_xx-12*rho_z^2*rho_y*a_x*a_y*a_z*rho_xxy-12*rho_z^2*rho_y*a_x*a_y*a_yz*rho_xx
-6*rho_z^2*rho_y*a_x*a_z*a_yy*rho_xx+6*rho_z^2*rho_y*a_x*a_z*rho_yy*a_xx+12*rho_z^2*rho_y*a_y
*a_z*a_xy*rho_xx-12*rho_z^2*rho_y*a_y*a_z*a_xx*rho_xy+2*rho_x^3*rho_zzz*a_y^3-2*rho_x^3*a_z^3
*rho_yyy-2*rho_y^3*rho_zzz*a_x^3+2*rho_y^3*a_z^3*rho_xxx+2*rho_z^3*rho_yyy*a_x^3-2*rho_z^3
*a_y^3*rho_xxx+12*rho*rho_x*rho_y*a_x*a_xz*rho_zz*a_yy-12*rho*rho_x*rho_y*a_x*rho_xz*a_yy*a_zz
-24*rho*rho_x*rho_y*a_y*a_z*a_xz*rho_xyz+12*rho*rho_x*rho_y*a_y*a_z*rho_xxy*a_zz+12*rho*rho_x
*rho_y*a_y*a_z*rho_yzz*a_xx-12*rho*rho_x*rho_y*a_y*a_xy*rho_xz*a_zz+12*rho*rho_x*rho_y*a_y
*a_xy*a_xz*rho_zz+12*rho*rho_x*rho_y*a_y*a_xz*rho_xz*a_yz-12*rho*rho_x*rho_y*a_y*a_xx*a_yz
*rho_zz+12*rho*rho_x*rho_y*a_y*a_xx*rho_yz*a_zz+12*rho*rho_x*rho_y*a_z*a_xy*rho_xz*a_yz
-12*rho*rho_x*rho_y*a_z*a_xy*rho_yz*a_xz-12*rho*rho_x*rho_y*a_z*rho_xz*a_xz*a_yy+12*rho*rho_x
*rho_y*a_z*a_yz*a_xx*rho_yz+12*rho*rho_x*rho_y*a_z*a_zz*a_yy*rho_xx-12*rho*rho_x*rho_y*a_z
*a_zz*rho_yy*a_xx-24*rho*rho_x*rho_z*a_x*a_y*rho_xyz*a_yz+12*rho*rho_x*rho_z*a_x*a_y*rho_xzz
*a_yy+12*rho*rho_x*rho_z*a_x*a_y*rho_xyy*a_zz-12*rho*rho_x*rho_z*a_x*a_xy*rho_yy*a_z
z+12*rho*rho_x*rho_z*a_x*a_xy*a_yz*rho_yz-12*rho*rho_x*rho_z*a_x*a_xz*rho_yz*a_yy+12*rho*rho_x
*rho_z*a_x*a_xz*a_yz*rho_yy+12*rho*rho_x*rho_z*a_x*rho_xy*a_yy*a_zz-12*rho*rho_x*rho_z*a_y*a_z
*rho_xxz*a_yy-12*rho*rho_x*rho_z*a_y*a_z*rho_yyz*a_xx+24*rho*rho_x*rho_z*a_y*a_z*rho_xyz*a_xy
+12*rho*rho_x*rho_z*a_y*a_xy*rho_yz*a_xz+12*rho*rho_x*rho_z*a_y*a_xy*rho_xy*a_zz-12*rho*rho_x
*rho_z*a_y*a_yz*rho_xy*a_xz-12*rho*rho_x*rho_z*a_y*a_yz*a_xx*rho_yz-12*rho*rho_x*rho_z*a_y
*a_zz*a_yy*rho_xx+12*rho*rho_x*rho_z*a_y*a_yy*a_xx*rho_zz-12*rho*rho_x*rho_z*a_z*a_xy*a_xz
*rho_yy-12*rho*rho_x*rho_z*a_z*a_xy*a_yz*rho_xy+12*rho*rho_x*rho_z*a_z*rho_xy*a_xz*a_yy
-12*rho*rho_x*rho_z*a_z*a_xx*rho_yz*a_yy+12*rho*rho_x*rho_z*a_z*a_xx*a_yz*rho_yy-12*rho*rho_z
*rho_y*a_x*a_y*rho_yzz*a_xx-12*rho*rho_z*rho_y*a_x*a_y*rho_xxy*a_zz+24*rho*rho_z*rho_y*a_x*a_y
*a_xz*rho_xyz+12*rho*rho_z*rho_y*a_x*a_z*rho_yyz*a_xx-24*rho*rho_z*rho_y*a_x*a_z*rho_xyz*a_xy
+12*rho*rho_z*rho_y*a_x*a_z*rho_xxz*a_yy-12*rho*rho_z*rho_y*a_x*a_xy*rho_xz*a_yz-12*rho*rho_z
*rho_y*a_x*a_xy*rho_xy*a_zz+12*rho*rho_z*rho_y*a_x*rho_xz*a_xz*a_yy+12*rho*rho_z*rho_y*a_x
*a_yz*rho_xy*a_xz-12*rho*rho_z*rho_y*a_x*a_yy*a_xx*rho_zz+12*rho*rho_z*rho_y*a_x*a_zz*rho_yy
*a_xx-12*rho*rho_z*rho_y*a_y*a_xy*rho_xz*a_xz+12*rho*rho_z*rho_y*a_y*a_xy*a_zz*rho_xx
-12*rho*rho_z*rho_y*a_y*rho_xx*a_yz*a_xz-12*rho*rho_z*rho_y*a_y*a_xx*rho_xy*a_zz+12*rho*rho_z
*rho_y*a_y*a_xx*rho_xz*a_yz+12*rho*rho_z*rho_y*a_z*a_xy*rho_xy*a_xz+12*rho*rho_z*rho_y*a_z
*a_xy*a_yz*rho_xx-12*rho*rho_z*rho_y*a_z*rho_xx*a_yy*a_xz+12*rho*rho_z*rho_y*a_z*a_xx*rho_xz
*a_yy-12*rho*rho_z*rho_y*a_z*a_xx*a_yz*rho_xy+24*rho_x*rho_z*rho_y*a_x*a_y*rho_xz*a_yz
-24*rho_x*rho_z*rho_y*a_x*a_y*rho_yz*a_xz+24*rho_x*rho_z*rho_y*a_x*a_z*rho_yz*a_xy-24*rho_x*
rho_z*rho_y*a_x*a_z*a_yz*rho_xy+24*rho_x*rho_z*rho_y*a_z*a_y*rho_xy*a_xz-24*rho_x*rho_z*rho_y
*a_z*a_y*a_xy*rho_xz+12*rho_x*rho_y^2*a_z^2*a_xy*rho_xz+6*rho_x*rho_z^2*a_x^2*rho_yz*a_yy
+6*rho_x*rho_z^2*a_x^2*rho_yyz*a_y-6*rho_x*rho_z^2*a_x^2*a_yz*rho_yy-6*rho_x*rho_z^2*a_x^2
*a_z*rho_yyy-12*rho_x*rho_z^2*a_x*rho_xyz*a_y^2+6*rho_x*rho_z^2*a_y^2*a_yz*rho_xx+12*rho_x
*rho_z^2*a_y^2*a_xy*rho_xz-6*rho_x*rho_z^2*a_y^2*a_z*rho_xxy-12*rho_x*rho_z^2*a_y^2*rho_xy
*a_xz-6*rho_x*rho_z^2*a_y^2*rho_yz*a_xx+6*rho_y^3*a_z*a_x*a_xx*rho_zz-6*rho_y^3*a_z*a_x*a_zz
*rho_xx-12*rho_z*rho_y^2*a_x^2*rho_xz*a_yz-6*rho_z*rho_y^2*a_x^2*a_y*rho_xzz-6*rho_z*rho_y^2
*a_x^2*rho_xy*a_zz+6*rho_z*rho_y^2*a_x^2*a_xy*rho_zz+12*rho_z*rho_y^2*a_x^2*rho_yz*a_xz
-12*rho_z*rho_y^2*a_x^2*a_z*rho_xyz+6*rho_z*rho_y^2*a_x*a_z^2*rho_xxy-6*rho_z*rho_y^2*a_y
*a_z^2*rho_xxx-6*rho_z*rho_y^2*a_z^2*a_xy*rho_xx+6*rho_z*rho_y^2*a_z^2*a_xx*rho_xy-12*rho_z^2
*rho_y*a_x^2*rho_yz*a_xy-6*rho_z^2*rho_y*a_x^2*a_xz*rho_yy+6*rho_z^2*rho_y*a_x^2*rho_xz*a_yy
+12*rho_z^2*rho_y*a_x^2*a_yz*rho_xy+12*rho_z^2*rho_y*a_x^2*a_y*rho_xyz+6*rho_z^2*rho_y*a_x^2
*a_z*rho_xyy-6*rho_z^2*rho_y*a_x*rho_xxz*a_y^2+6*rho_z^2*rho_y*a_y^2*a_z*rho_xxx-6*rho_z^2
*rho_y*a_y^2*a_xx*rho_xz+6*rho_z^2*rho_y*a_y^2*a_xz*rho_xx+6*rho_z^3*a_y*a_x*a_yy*rho_xx
-6*rho_z^3*a_y*a_x*rho_yy*a_xx+6*rho_z^3*a_y^2*a_xx*rho_xy
\end{verbatim}
\enlargethispage{\baselineskip}

}\label{lastpage}
\end{document}